\documentclass{article}

\usepackage{amsmath}
\usepackage{amsthm}
\usepackage{amsfonts}
\usepackage{amsbsy}
\usepackage{amssymb}
\usepackage[initials]{amsrefs}
\newtheorem{theorem}{Theorem}
\newtheorem{proposition}{Proposition}
\newtheorem{corollary}{Corollary}

\newtheorem{lemma}{Lemma}

\newcommand{\R}{{\mathbb R}}
\newcommand{\Z}{{\mathbb Z}}
\newcommand{\C}{{\mathbb C}}
\newcommand{\set}[2]{ \left\{ #1 \ \left| \ #2 \right. \right\} }
\newcommand{\SL}{{\mathrm{SL}}}
\newcommand{\GL}{{\mathrm{GL}}}
\newcommand{\ms}{\cdot {\boldsymbol \partial}}
\newcommand{\mso}{{\boldsymbol \partial}}
\newcommand{\MS}{{\mathbb M}}

\title{Geometric averaging operators and nonconcentration inequalities}
\author{Philip T.~Gressman\footnote{This work was partially supported by NSF grants DMS-1361697 and DMS-1764143.}}
\date{\today}

\begin{document}
\maketitle
\begin{abstract}
This paper is devoted to a systematic study of certain geometric integral inequalities which arise in continuum combinatorial approaches to $L^p$-improving inequalities for Radon-like transforms over polynomial submanifolds of intermediate dimension.  The desired inequalities relate to and extend a number of important results in geometric measure theory.
\end{abstract}

\tableofcontents

\section{Introduction}
\subsection{Main results}
Suppose that $\gamma(t,x)$ is a polynomial map from $\R^{n} \times \R^{N_2}$ into $\R^{N_1}$ with  $r := N_1 -n > 0$ and that $\widetilde \Omega$ is some Borel measurable subset of $\R^{n} \times \R^{N_2}$. To this $\gamma$ and $\widetilde \Omega$, one may associate the Radon-like operator
\begin{equation}
T f(x) := \int_{\R^n} f(\gamma(t,x)) \chi_{\widetilde \Omega} (t,x) dt, \label{theop}
\end{equation}
which may be informally regarded as averaging functions $f$ on $\R^{N_1}$ over the family of sets $\{ \Sigma_x\}_{x \in \R^{N_2}}$ given by 
\[ \Sigma_x := \set{ \gamma(t,x) \in \R^{N_1}}{ t \in \R^{n}, \ (t,x) \in \widetilde \Omega}. \]
 The main result of this paper regarding the operator \eqref{theop} is the following:
\begin{theorem}
Suppose $N_2 = rk$ for some positive integer $k$.  Let $\omega$ be the $r$-form \label{radon}
\begin{equation}
\begin{split} \omega & (t,x) :=  \! \! \! \label{normalform} \\ & \sum_{1 \leq i_1 < \cdots < i_r \leq N_2} \! \! \det \left[ \! \! \begin{array}{cccc} \frac{\partial \gamma}{\partial x_{i_1}}(t,x) & \! \! \cdots & \! \! \! \frac{\partial \gamma}{\partial x_{i_r}}(t,x) & \! \! \frac{\partial \gamma}{\partial t}(t,x)  \end{array} \! \! \right] d x_{i_1} \wedge \cdots \wedge dx_{i_r},
\end{split} 
\end{equation}
where each ${\partial \gamma}/{\partial x_{i_j}}$ is an $N_1 \times 1$ column matrix of partial derivatives, ${\partial \gamma}/{\partial t}$ is the $N_1 \times n$ Jacobian matrix of $\gamma$ with respect to $t$, and the determinant is that of the $N_1 \times N_1$ square matrix formed by concatenation. For each $x \in \R^{N_2}$, let\footnote{Note that the ratio of forms in the definition of $\Phi_x$ is a well-defined real number because both numerator and denominator belong to the same one-dimensional vector space of $N_2$-forms on $\R^{N_2}$.}
\begin{equation} \Phi_x(t_1,\ldots,t_k) := \frac{\omega(t_1,x) \wedge \cdots \wedge \omega(t_k,x)}{dx_1 \wedge \cdots \wedge dx_{N_2}}. \label{thephi} \end{equation}
Fix any real $s,\delta > 0$ and suppose that  $\widetilde \Omega \subset \R^{n} \times \R^{N_2}$ is a Borel set such that
\begin{equation} \int_{E^k}  |\Phi_x(t_1,\ldots,t_k)| dt_1 \cdots d t_k \geq \delta |E|^{k+s} \label{mainhyp} \end{equation}
for every point $x \in \R^{N_2}$ and every Borel $E \subset \R^{n}$ such that $E \times \{x\} \subset \widetilde \Omega$, where $|E|$ denotes the Lebesgue measure of $E$. Then the Radon-like operator \eqref{theop}
satisfies the inequality
\begin{equation} || T \chi_F ||_{L^{k+s}(\R^{N_2})} \lesssim  \delta^{-\frac{1}{k+s}} |F|^\frac{k}{k+s} \label{mainconc} \end{equation}
for all Borel sets $F \subset \R^{N_1}$, with the notation ``$\lesssim$'' indicating the presence of an implicit multiplicative factor. In this case, the factor depends only on $(n,N_1,N_2,s, \deg \gamma)$.\end{theorem}

The technical structure of the proof is built on the change of variables formula, similar to various earlier approaches \cites{gressman2013,gressman2015} in the spirit of combinatorial/continuum incidence methods developed by Christ \cite{christ1998}. Christ's technique, based on ideas of Bourgain \cites{bourgain1986,bourgain1991}, Wolff \cites{wolff1995,wolff1997}, Schlag \cite{schlag1997}, and others, has, since its development twenty years ago, had an impact on the subject of harmonic analysis which is difficult to overstate. It has influenced and inspired work of Bennett, Carbery, and Wright \cite{bcw2005}, Dendrinos, Laghi, and Wright \cite{dlw2009}, Erdo\u{g}an and R.~Oberlin \cite{eo2008}, Hickman \cite{hickman2016}, D.~Oberlin \cite{oberlin2000II}, Stovall \cites{stovall2011,stovall2014}, Tao and Wright \cite{tw2003}, and many others.

  When $r=1$, the operator \eqref{theop} integrates over hypersurfaces and the integral on the left-hand side of \eqref{mainhyp} reduces to a multilinear determinant functional \cite{gressman2010}. In this case it is known that for fixed $x$, the inequality \eqref{mainhyp} is satisfied if and only if the Lebesgue measure $dt$ on the submanifold $\Gamma_x \subset \R^{N_2}$ parametrized by $t \mapsto \omega(t,x)$ satisfies D.~Oberlin's affine curvature condition, meaning that  \begin{equation} \int \chi_{R \cap \widetilde \Omega} ( \omega(t,x)) dt \lesssim |R|^{\frac{1}{s}} \label{oberlin0} \end{equation}
  for all boxes $R$ with arbitrary orientations and eccentricities, with an implicit constant which is independent of $R$. The condition \eqref{oberlin0} is called affine because the implicit constant does not change when $\Gamma_x$ is acted on by an equiaffine\footnote{The prefix ``equi-'' specifies those affine transformations which preserve Lebesgue measure.} transformation and is regarded as a curvature condition because it necessarily fails when $\Gamma_x$ lies in any affine hyperplane. The question of whether \eqref{oberlin0} is satisfied for a given $\omega(t,x)$ is surprisingly difficult to solve and systematic approaches have only recently become available \cite{gressman2017}. When $r > 1$, the situation is even more difficult, as there are no previously-known analogues of the Oberlin affine curvature condition which apply to \eqref{mainhyp}.
  
To address the inherent difficulties of the case $r > 1$, this paper is devoted primarily to the general study of functionals  of the form 
\begin{align}
{\mathcal A}  (E) & :=   \int_{E^k}   |\Phi(x_1,\ldots,x_k)|  d \mu(x_1) \cdots d \mu(x_k) 
 \label{mainobj} 
 \end{align}
 and
 \begin{align}
{\mathcal S}  (E) & := \sup_{(x_1,\ldots,x_k) \in E^k}  |\Phi(x_1,\ldots,x_k)|
 \label{mainobj2}
\end{align}
where the sets $E$ range over all Borel subsets of some domain $\Omega \subset \R^{n}$ and the measure $\mu$ is a nonnegative Borel measure. Functionals of the forms \eqref{mainobj} and \eqref{mainobj2} will be called nonconcentration functionals since they quantify the extent to which product sets $E^k$ fail to lie in the zero set of $\Phi$.
Outside of the context of Theorem \ref{radon}, $\Phi : \Omega^k \rightarrow \R^m$ will be taken to be any polynomial in $(x_1,\ldots,x_k)$ which vanishes to order $q \geq 1$ on the diagonal $\Delta := \set{(x_1,\ldots,x_k) \in \Omega^k}{x_1 = \cdots = x_k}$, meaning that all partial derivatives  of order less than $q$ vanish identically on $\Delta$ and some partial derivative of order $q$ is nonzero at some point of $\Delta$. When $m > 1$, the absolute values $|\cdot|$ are to be understood as some fixed but otherwise arbitrary norm on $\R^m$. The general question to be answered is to determine when one has inequalities of the form
\begin{align}
 \mathcal A(E) & \geq c_{\mu,s} \left[ \mu(E) \right]^{k + s} \label{anorm}
 \end{align}
 and
 \begin{align}
\mathcal S(E) & \geq c_{\mu,s}'  \left[ \mu(E) \right]^{s} \label{snorm}
\end{align}
for all Borel sets $E \subset \Omega$, where $s > 0$ is a fixed real number and $c_{\mu,s}$ and $c_{\mu,s}'$ are nonnegative constants which do not depend on $E$. The cases $c_{\mu,s} = 0$, $c_{\mu,s}' = 0$, and $\mu = 0$ are uninteresting; to avoid these exceptions, a nonnegative Borel measure $\mu$ on $\Omega$ will be said to satisfy \eqref{anorm} or \eqref{snorm} nontrivially when $\mu$ is not the zero measure and the corresponding inequality holds with a strictly positive constant. 
Both \eqref{anorm} and \eqref{snorm} will be called nonconcentration inequalities.

The first significant result for nonconcentration inequalities establishes the fundamental equivalence of \eqref{anorm} and \eqref{snorm}:
\begin{theorem}
For any nonnegative Borel measure $\mu$ and any $s > 0$, $\mu$ satisfies \eqref{anorm} with positive constant if and only if $\mu$ satisfies \eqref{snorm} with positive constant. Moreover, if one defines $||\mathcal A||_{\mu,s}$ to be the supremum of all nonnegative $c_{\mu,s}$ such that \eqref{anorm} holds for all Borel sets $E \subset \Omega$ and likewise defines $||\mathcal S||_{\mu,s}$ to be the supremum of all $c_{\mu,s}'$ satisfying \eqref{snorm} for all Borel $E \subset \Omega$,
then \label{aseq}
\begin{equation} || \mathcal S||_{\mu,s} \geq || \mathcal A||_{\mu,s} \gtrsim || \mathcal S||_{\mu,s} \label{aseqeq}, \end{equation}
where the implicit constant depends only on on $(n,k,s, \deg \Phi)$.
\end{theorem}
The value of Theorem \ref{aseq} is that the nonconcentration functional $\mathcal S$ is generally much easier to calculate and estimate than $\mathcal A$. In particular, it is possible to characterize existence of nontrivial measures $\mu$ satisfying \eqref{snorm} in terms of a geometric measure-theoretic generalization of Hausdorff measure and a corresponding generalization of Frostman's Lemma. In in the key ``dimension'' for this measure, it is also possible to deduce detailed information about the Radon-Nykodym derivative of this generalized Hausdorff measure with respect to Lebesgue measure. When combined with Theorem \ref{aseq}, this gives an explicit calculation which can be carried out to verify the hypothesis \eqref{mainhyp}. Some of the most important results in this direction are summarized in the following theorem. 
\begin{theorem}

For any Borel set $E \subset \Omega$ and any $\sigma > 0$, the $\sigma$-dimensional weighted $\Phi$-Hausdorff measure of $E$ is defined to equal the quantity \label{bestmeasthm} 
\begin{equation}
\begin{split}
 \lambda^\sigma_\Phi(E)  := \lim_{\delta \rightarrow 0^+} \! \inf \left\{  \sum_{i=1}^\infty c_i \left[ \mathcal S(E_i) \right]^\sigma  ~ \! \right| ~ \chi_E \leq \sum_{i=1}^\infty c_i \chi_{E_i}, & \label{weightedhm0} \\
      c_i \geq  0 \mbox{ and }  \mathrm{diam}(E_i) & \leq \delta \mbox{ for all i}  \left. \vphantom{  \left[  \sum_i \mathcal S(E_i) \right]^s } \right\}. 
  \end{split}
 \end{equation}
  Then the following statements are true:
\begin{enumerate}
\item If  $\sigma > n/q$, then $\lambda^{\sigma}_\Phi(\Omega) = 0$. There are no Borel measures $\mu$ satisfying \eqref{snorm} nontrivially when $s = 1/\sigma$.
\item If $\sigma \leq n/q$, then there is a Borel measure $\mu$ satisfying  \eqref{snorm} nontrivially with $s = 1/\sigma$ if and only if $\lambda^{\sigma}_{\Phi}(\Omega) > 0$.
\item If $\sigma = n/q$, $\lambda^{\sigma}_\Phi$ is absolutely continuous with respect to Lebesgue measure and there is an explicit estimate (see \eqref{hausdens2}) for the pointwise magnitude of the Radon-Nykodym derivative. Moreover
\begin{equation} \mathcal S(E) \gtrsim \left[ \lambda^{\frac{n}{q}}_\Phi(E) \right]^{\frac{q}{n}} \gtrsim || \mathcal S||_{\mu,\frac{q}{n}} \left[ \mu(E) \right]^{\frac{q}{n}} \label{twoway} \end{equation}
for any Borel set $E$ and any nonnegative Borel measure $\mu$ satisfying \eqref{snorm}, with implicit constants depending only on $(n,k, q, \deg \Phi)$. In other words, the measure $\lambda^{n/q}_\Phi$ satisfies \eqref{snorm} itself and is, up to a multiplicative constant, the largest such measure. 
\end{enumerate}
\end{theorem}

\subsection{Examples}

It is worthwhile to briefly examine the implications of Theorem \ref{bestmeasthm} in some familiar and unfamiliar settings. \label{examples1}

{\bf Example 1 (Hausdorff measure).}
When $\Phi(x,y) := x-y$ for $x,y \in \R^n$, $\mathcal S(E)$ is the diameter of $E$ and $\lambda^\sigma_\Phi$ is equal to the classical $\sigma$-dimensional Hausdorff measure $\mathcal H^\sigma$ (see Federer \cite[2.10.24]{federer1969}). The order of vanishing $q$ is simply $1$. The first inequality of \eqref{twoway} states that
\begin{equation*} |E| \lesssim \left[ \mathrm{diam}(E) \right]^n. \end{equation*}
In its sharp form with optimal constant, this is known as the isodiametric inequality \cite[2.10.33]{federer1969}. Likewise, if $\mu$ is any nonnegative Borel measure satisfying
\begin{equation} \mu(E) \lesssim \left[ \mathrm{diam}(E) \right]^n \label{dimension} \end{equation}
for every Borel set $E \subset \Omega$, then
 \eqref{twoway} implies that $\mu(E) \lesssim |E|$. Thus Lebesgue measure on $\R^n$ is, up to a constant, the largest measure on $\R^n$ satisfying an isodiametric inequality \eqref{dimension}. It should also be noted that the inequality \eqref{dimension} is, modulo the constant, equivalent to the upper Ahlfors regularity condition \[ \mu(B_r(x)) \lesssim r^n \] for all Euclidean balls $B_r(x) \subset \R^n$, since every set $E$ of bounded diameter is contained in a ball of comparable diameter by virtue of Jung's Theorem \cite{federer1969}.

{\bf Example 1\textsuperscript{\ensuremath{\prime}} (Hausdorff measure).}
To generalize the first example, suppose that  $\gamma : \R^{p} \rightarrow \R^n$, $p < n$, is any locally injective polynomial function and set $\Phi(x,y) := \gamma(x) - \gamma(y)$. Locally the measure $\lambda^{p}_\Phi$ on $\R^p$ pushes forward to equal exactly the $p$-dimensional Hausdorff measure on $\R^n$ restricted to the image of $\gamma$. Because the multiplicity of images of $\gamma$ is bounded in terms of the degree, the measures must be comparable globally as well. The order of vanishing $q$ is still $1$, and by \eqref{twoway}, it follows that the $p$-dimensional Hausdorff measure on the image of $\gamma$ also satisfies an isodiametric inequality on $\R^n$, i.e.,
\begin{equation} \mathcal H^p( \gamma(E)) \lesssim [ \mathrm{diam} ( \gamma(E)) ]^p. \label{isodiam2} \end{equation}
Such an inequality can only hold in general because $\gamma$ is polynomial; if $\gamma$ were merely $C^\infty$ it  is easy to construct a highly oscillatory curve, for example, with infinite length inside a ball of finite radius. It is also worth noting that up to multiplicative constants, the measure $\mathcal H^p$ restricted to the image of $\gamma$ is essentially the largest measure satisfying the $p$-dimensional upper Ahlfors regularity condition equivalent to \eqref{isodiam2}.

{\bf Example 2 (Determinantal measure).} An interesting nontrivial example \label{detex} on the space of $n \times n$ matrices is to set $\Phi(A_1,A_2) := \det (A_1 - A_2)$ for any $A_1, A_2 \in \R^{n \times n}$. The order of vanishing $q$ equals $n$. Using the the estimate \eqref{hausdens2} for the magnitude of the Radon-Nykodym derivative $d \lambda^n_\Phi / dx$, it will be shown (see Proposition \ref{calcprop}) that $\lambda_\Phi^{n}$ is comparable to Lebesgue measure on $\R^{n \times n}$. Thus, the first inequality of \eqref{twoway} becomes a determinantal isodiametric inequality for subsets of $\R^{n \times n}$, namely,
\[ |E| \lesssim  \left[ \sup_{A,A' \in E} |\det (A - A')| \right]^n \]
for all Borel sets $E \subset \R^{n \times n}$. The implications of this inequality for a corresponding Radon-like operator are detailed in Section \ref{examples2}.

{\bf Example 3 (Affine measure).}
For $\gamma$ as in Example 1$^\prime$, let
 \[ \Phi(x_1,\ldots,x_{n+1}) := \det(\gamma(x_1) - \gamma(x_{n+1}),\ldots,\gamma(x_n) - \gamma(x_{n+1})),\] where the determinant of an ordered list of $n$ vectors in $\R^n$ is defined to equal the determinant of the $n \times n$ matrix whose $j$-th column contains the ordered coordinates of the $j$-th vector in the standard basis.
The measure $\lambda^{\sigma}_\Phi$ pushes forward to a measure on the graph of $\gamma$ which is
 is dominated by D.~Oberlin's affine measure of dimension $n \sigma$ \cite{oberlin2003} up to a uniform multiplicative constant; while it is not clear that these two measures are comparable in all cases, it is a consequence of later arguments in this paper that the measures must be comparable when $\sigma = p/q$. For this particular value of $\sigma$, $\lambda^{\sigma}_\Phi$ is comparable to the recently-defined affine hypersurface measure \cite{gressman2017}, which is the optimal measure satisfying Oberlin's affine curvature condition
 \begin{equation} \mu(R) \lesssim |R|^{\frac{q}{p}} \label{oberlin2} \end{equation}
 for all boxes $R \subset \R^n$ of arbitrary orientation. Similar to the Hausdorff measure and the upper Ahlfors regularity condition, the Oberlin condition \eqref{oberlin2} is in fact equivalent to the {\it a priori} stronger inequality \eqref{snorm} (see Section \ref{basicgmt}).
 
 {\bf Example 4 (Projective Measure on Forms).}
When the underlying space is taken to be the decomposable\footnote{Here ``decomposable'' means expressible as an $r$-fold wedge product of $1$-vectors.} $r$-vectors in $\Lambda^{r}(\R^{rk})$ for positive integers $r$ and $k$,
let \begin{equation}
\begin{split}
\mathcal {P}^\sigma(E) := \lim_{\delta \rightarrow 0^+} \! \inf  \left\{ \sum_{i=1}^\infty c_i\sup_{\omega_1,\ldots, \omega_k \in E_i} \left|\frac{\omega_1 \wedge \cdots \wedge \omega_k}{e_1 \wedge \cdots \wedge e_{rk}} \right|^\sigma \   \right|  \  \chi_E  \leq & \sum_{i=1}^\infty c_i \chi_{E_i},  \\
c_i \geq 0 \mbox{ and } \mathrm{diam} (E_i) \leq \delta & \mbox{ for all } i \left. \vphantom{ \sum_{i=1}^\infty \sup_{\omega_1,\ldots, \omega_k \in E_i} \left|\frac{\omega_1 \wedge \cdots \wedge \omega_k}{e_1 \wedge \cdots \wedge e_{rk}} \right|^\sigma} \right\} 
\end{split}
\end{equation}
(where diameter is with respect to any metric inducing the usual topology).
The form $\omega(t,x)$ defined by \eqref{normalform} is always decomposable (see Sections \ref{radonsec} and \ref{examples2});
if $t \mapsto \omega(t,x)$ is locally injective for each $x$, then the push forward of the measure $\lambda_{\Phi_x}^\sigma$ on $\R^n$ to the graph of $\omega(\cdot,x)$ will be comparable to the restriction of $P^{\sigma}$ to the same graph.
If $q$ is the smallest integer such that $\Phi_x(t_1,\ldots,t_k)$ vanishes to order $q$ on the diagonal for some $x$, then setting
\[ \widetilde \Omega := \set{ (t,x) \in \R^n \times \R^{N_2}}{  \frac{d \lambda^{\frac{n}{q}}_{\Phi_x}}{dt} (t) \geq c \delta^\frac{n}{q}} \]
for an appropriate constant $c$ depending only on $(n,q,N_1,N_2,\deg \gamma)$ yields the inequality \eqref{mainhyp} with $s = q/n$ by Theorem \ref{aseq} together with the fact that
\[ \mathcal S(E \cap \Omega_x) \gtrsim \left[ \lambda^{\frac{n}{q}}_{\Phi_x}(E \cap \Omega_x) \right]^\frac{q}{n} \geq \left[ c \delta^{\frac{n}{q}} |E \cap \Omega_x| \right]^{\frac{q}{n}} \]
when $\Omega_x$ is the set where the Radon-Nykodym derivative $d \lambda^{n/q}_{\Phi_x} / dt$ exceeds $c \delta^{n/q}$.

\subsection{Structure of the paper}

Section \ref{radonsec} is a self-contained proof of Theorem \ref{radon} using a combinatorial approach much like earlier work on uniform sublevel Radon-like inequalities and averages over $n$-dimensional submanifolds of $\R^{2n}$ \cites{gressman2013,gressman2015}.
Section \ref{gmtsec} contains a proof of Theorem \ref{aseq} using elementary convex geometry as via Lemma \ref{c0lemma},a n earlier version of which appears in work on affine submanifold measures \cite{gressman2017}. This section also contains some basic GMT observations about $\Phi$-Hausdorff and weighted $\Phi$-Hausdorff measures which will be used in the proof of Theorem \ref{bestmeasthm}. In particular, Section \ref{basicgmt} contains a proof of the relevant generalization of Frostman's lemma, which is a rather direct reinterpretation of Howroyd's proof as appearing in Mattila's book \cite{mattila}.
Section \ref{mainthmsec} provides the bulk of the proof of Theorem \ref{bestmeasthm}. The case $\sigma < n/q$ is essentially an immediate consequence of Lemma \ref{frostmanL}, while the case $\sigma \geq n/q$ relies on a scaling argument to show that $\Phi$-Hausdorff measure of dimension $\sigma$ must be absolutely continuous with respect to Lebesgue measure and to consequently estimate the Radon-Nykodym derivative. At this point, the remaining portions of Theorem \ref{bestmeasthm} are reduced to establishing Theorem \ref{betterthm}, which gives an explicit construction for any $s$ of a measure (possibly zero) satisfying \eqref{snorm}.
The proof of Theorem \ref{betterthm} is then reduced to proving Lemma \ref{multisystem} (see also \cite{gressman2017}), which is the content of Section \ref{mainlemmasec}. As a part of the proof of Lemma \ref{multisystem}, Section \ref{mainlemmasec} also identifies the underlying intrinsic geometric objects which play an important algebraic role in the lemma and relate closely to earlier geometric sublevel set estimates \cite{gressman2010II}.
Finally, Section \ref{examples2} gives some example applications of Theorem \ref{radon} which correspond to the GMT examples from Section \ref{examples1}.

\section{Proof of Theorem \ref{radon}}
\label{radonsec}
\begin{proof}[Proof of Theorem \ref{radon}]
As defined in the introduction, suppose that $\gamma(t,x)$ is a polynomial map from $\R^n \times \R^{N_2}$ into $\R^{N_1}$. Let $r := N_1 - n$, and suppose that $N_2 = r k$ for some integer $k$.  The basic structure of this proof is to estimate the quantity
\begin{equation} Q(F) := \int_{\R^{N_2}}  \int_{(\R^n)^k} |\Phi_x(t_1,\ldots,t_k)| \prod_{j=1}^k \chi_F(\gamma(t_j,x)) \chi_{\widetilde \Omega}(t_j,x) dt_1 \cdots dt_k dx \label{centralq} \end{equation}
from below and above, where $\Phi_x(t_1,\ldots,t_k)$ is defined to be the Jacobian determinant of the map $(x,t_1,\ldots,t_k) \mapsto (\gamma(t_1,x),\ldots, \gamma(t_k,x))$. The main upper bound for $Q(F)$ comes from the change of variables formula and B\'ezout's Theorem: for any $(u_1,\ldots,u_k) \in (\R^{N_1})^k$, since $N_1k = N_2 + nk$, B\'ezout's Theorem guarantees that the number of connected components in $\C^{N_1k}$ of the solution set of the system of equations 
\begin{equation} (\gamma(t_1,x),\ldots,\gamma(t_k,x)) = (u_1,\ldots,u_k) \label{system} \end{equation} is at most the product of the degrees of the polynomials (see Fulton \cite[Chapter 8, Section 4]{fulton1984}).
This means that the number of real solutions of the system where the Jacobian is nonvanishing cannot exceed this same upper bound, since the nonvanishing of the Jacobian at a real solution guarantees that such a solution will be isolated in complex space as well. Now by the change of variables formula, if the number of solutions $(x,t_1,\ldots,t_k)$ of the system \eqref{system} inside the domain of the integral $Q(F)$ is never greater than $N$ for any choice of $(u_1,\ldots,u_k)$, then
\begin{equation} Q(F) \leq N \int_{(\R^{N_1})^k} \prod_{j=1}^k \chi_F(u_j) d u_1 \cdots d u_k = N |F|^k. \label{covf} \end{equation}
Without loss of generality, it may be assumed that Jacobian determinant is nonvanishing at every counted solution of the system (since the integral on the set where $|\Phi_x(t_1,\ldots,t_k)| = 0$ is necessarily zero), i.e., $N$ need only bound the number of isolated solutions of \eqref{system} for a given right-hand side $(u_1,\ldots,u_k)$, which B\'{e}zout's Theorem guarantees is bounded by the product of degrees.

To estimate \eqref{centralq} from below,
recall the definition \eqref{normalform} of the form $\omega$.  The key fact to establish is that the functional $\Phi_x$ is indeed the Jacobian determinant of the map $(x,t_1,\ldots,t_k) \mapsto (\gamma(t_1,x),\ldots, \gamma(t_k,x))$, i.e., that
\begin{equation} \Phi_x(t_1,\ldots,t_k) := \det \frac{\partial ( \gamma(t_1,x),\ldots, \gamma(t_k,x))}{\partial (x,t_1,\ldots,t_k)} = \frac{\omega (t_1,x)  \wedge \cdots \wedge \omega (t_k,x)}{dx_1 \wedge \cdots \wedge dx_{N_2}}. \label{jacob} \end{equation}
To prove \eqref{jacob}, first observe that the Jacobian matrix has block structure
\begin{equation} \left[ \begin{array}{ccccc} 
\frac{\partial \gamma}{\partial x} (t_1,x) & \frac{\partial \gamma}{\partial t}(t_1,x) & 0 & \cdots & 0  \\
\vdots & 0 & \ddots & \ddots & \vdots \\
\frac{\partial \gamma}{\partial x} (t_{k-1},x) & \vdots & \ddots & \frac{\partial \gamma}{\partial t}(t_{k-1},x) & 0 \\
\frac{\partial \gamma}{\partial x} (t_{k},x) & 0 & \cdots & 0 & \frac{\partial \gamma}{\partial t}(t_{k},x)
\end{array} \right] \label{jacobblock} \end{equation}
where $\partial \gamma / \partial x$ is an $N_1 \times N_2$ block of partial derivatives of $\gamma$ (with the coordinates of $\gamma$ corresponding to rows and the partial derivatives in the coordinate directions of $x$ corresponding to columns) and $\partial \gamma / \partial t$ is a corresponding $N_1 \times n$ block of partial derivatives.
 To simplify the determinant of the matrix \eqref{jacobblock}, label the coordinates of $t_j$ as $(t_{j1},\ldots,t_{jn})$. It will be necessary to use the identity
\begin{align}
( a_{11} & d x_1 + \cdots + a_{1N_2} dx_{N_2} + b_{11} d t_{j1} + \cdots + b_{1n} d t_{jn} ) \wedge \cdots \nonumber \\
& \cdots \wedge ( a_{N_1 1} d x_1 + \cdots + a_{N_1 N_2} dx_{N_2} + b_{N_1 1} d t_{j1} + \cdots + b_{N_1 n} d t_{jn}) \label{wedgep} \\
& \qquad = \omega_j \wedge d t_{j1} \wedge \cdots \wedge d t_{jn} + E_j \nonumber
\end{align}
where one defines
\[ \begin{split}
 \omega_j & := \\ & \mathop{\sum_{i_1,\ldots,i_{r}=1}^{N_2}}_{i_1 < \cdots < i_{r}} \! \! \! \det \left[ \begin{array}{cccccc} a_{1i_1} & \cdots &  a_{1 i_{r}} & b_{11} & \cdots & b_{1n} \\
\vdots & \ddots & \vdots & \vdots & \ddots & \vdots \\
a_{N_1 i_1} & \cdots & a_{N_1 i_{r}} & b_{N_11} & \cdots & b_{N_1 n} \end{array} \right] dx_{i_1} \wedge \cdots \wedge d x_{i_{r}} 
\end{split} \]
and observes of the remainder $E_j$ that it is spanned by all $N_1$-fold wedge products of $dx_1,\ldots,$ $dx_{N_2}$, $dt_{j1},\ldots,$ $dt_{jn}$ which omit $d t_{ji}$ for at least one index $i \in \{1,\ldots,n\}$. The proof of the identity is essentially immediate after observing that when computing the correct coefficient of $dx_{i_1} \wedge \cdots \wedge dx_{i_r}$ in $\omega_j$, it suffices to assume that $a_{ji} = 0$ for $i \neq i_1,\ldots,i_{r}$. 

To use the identity \eqref{wedgep}, first express the determinant as the coefficient of $dx_1 \wedge \cdots \wedge dx_{N_2} \wedge d t_{11} \wedge \cdots \wedge dt_{1n} \wedge \cdots \wedge dt_{k1} \wedge \cdots \wedge dt_{kn}$ in an $(N_2 + kn)$-fold wedge product of one forms with coefficients drawn from the rows of the block-form matrix \eqref{jacobblock}. The wedge of the forms in the $j$-th block of rows is given by \eqref{wedgep} when each coefficient $a_{ii'}$ is replaced the $(i,i')$-entry of the matrix $({\partial \gamma}/{\partial x})(t_j,x)$ and each coefficient $b_{ii'}$ is replaced the $(i,i')$-entry of the matrix $({\partial \gamma}/{\partial t})(t_j,x)$. In particular, this yields the identity $\omega_j = \omega(t_j,x)$. To compute the Jacobian determinant \eqref{jacob}, it suffices to take the wedge of the expressions \eqref{wedgep} over $j=1,\ldots,k$ and show that the remainders $E_j$ do not influence the coefficient of $dx_1 \wedge \cdots \wedge dx_{N_2} \wedge dt_{11} \wedge \cdots \wedge dt_{1n} \wedge \cdots \wedge dt_{k1} \wedge \cdots \wedge dt_{kn}$. Because the variables $t_j$ appear only in the $j$-th block of rows, there is only one way for $dt_{j1} \wedge \cdots \wedge dt_{jm}$ to be a factor in the full wedge product: it must appear explicitly in a corresponding term of \eqref{wedgep}. In other words, when taking the wedge over all $j$, any wedge product including an $E_j$ will not contain all $n$ factors $dt_{j1},\ldots,d t_{jn}$. In the place of the missing $d t_{ji}$, every term of $E_j$ must necessarily contain more than $r$ factors drawn from $dx_{1},\ldots,d x_{N_2}$. Since every term of the wedge product \eqref{wedgep} must contain {\it at least} $r$ factors drawn from $dx_1,\ldots,dx_{N_2}$, it follows by the pigeonhole principle that in the full $k$-fold wedge product representing the determinant \eqref{jacobblock}, when expanded by multilinearity, any term including $E_j$ must be expressible as a sum of wedge products with at least one duplicate $dx_i$. Thus \eqref{jacob} must hold.

It is worth pausing briefly to make the observation that $\omega$ must be decomposable. First note that the form $\omega$ as defined by \eqref{normalform} is independent of the chosen coordinate systems on $\R^{N_2}$ and $\R^{n}$. 
If $t \mapsto \gamma(t,x)$ does not have injective differential, then $\omega(t,x)$ vanishes. Thus, when $\omega$ is nonzero, the dimension of the quotient $\R^{N_1}$ modulo the image of the differential $d_t \gamma(t,x)$ always has dimension $r = N_1 - n$. The image of the differential $d_x \gamma(t,x)$ in this quotient space is therefore at most $r$-dimensional, meaning that whenever $\omega(t,x)$ is not zero, it is always possible to choose a coordinate system near any given $x$ for which ${\partial \gamma}/{\partial x_i}$ belongs to the span of the $t$ partial derivatives of $\gamma$ whenever $i > r$. Computing the form \eqref{normalform} in these coordinates shows that $\omega$ must be a multiple of $dx_1 \wedge \cdots \wedge dx_r$ and is therefore decomposable. Moreover,  it follows that
$\omega(t_1,x) \wedge \omega(t_2,x)$
vanishes to at least order $r$ when $t_1 = t_2$ and $\omega(t_1,x) \neq 0$. This then implies that $\Phi_x(t_1,\ldots,t_k)$ vanishes to order at least $r (k-1)$ on the diagonal $\Delta$ at all points where $\omega(t,x) \neq 0$. 

Returning to \eqref{centralq}, 
fix a Borel measurable set $F \subset \R^{N_1}$. By \eqref{covf},
\[ \int \left| \Phi_x(t_1,\ldots,t_k) \right| \left[ \prod_{j=1}^k \chi_F(\gamma(t_j,x)) \chi_{\widetilde \Omega}(t_j,x) \right] dx dt_1 \cdots dt_k \lesssim |F|^k, \]
where the implicit constant can be taken to equal the maximum number of isolated solutions $(x,t_1,\ldots,t_k)$ of the system $(\gamma(t_1,x),\ldots,\gamma(t_k,x)) = (u_1,\ldots,u_k)$ as $u_1,\ldots,u_k$ range over $\R^{N_1}$. Defining $F_x \subset \R^m$ to equal
\[ F_x := \set{t \in \R^n}{ \gamma(t,x) \in F, \ (t,x) \in \widetilde \Omega} \]
(which will be a Borel subset of $\R^n$ since $\gamma$ is a continuous function of $t$), it follows by Fubini that
\[ \int \left[ \int_{F_x^k} \left| \Phi_x(t_1,\ldots,t_k) \right| dt_1 \cdots d t_k \right] dx \lesssim |F|^k. \]
By the main hypothesis \eqref{mainhyp} of Theorem \ref{radon}, it must be the case that
\begin{equation} \int \delta |F_x|^{k+s} dx \leq \int \left[ \int_{F_x^k} \left| \Phi_x(t_1,\ldots,t_k) \right| dt_1 \cdots d t_k \right] dx \lesssim |F|^k \label{lastobs} \end{equation}
since for each $x$, $F_x \times \{x\} \subset \widetilde \Omega$. However, by the definition \eqref{theop} of the Radon-like operator $T$,
\[ |F_x| = T \chi_F(x) \]
for each $x$. Inserting this equality into \eqref{lastobs} and raising both sides to the power $1/(k+s)$ gives the conclusion \eqref{mainconc} of Theorem \ref{radon}.
\end{proof}

As a final remark concerning the proof, it should be noted that the constraint that $r = N_1 -n$ divides $N_2$ is only used in proving the upper bound for \eqref{centralq} via the change of variables formula. As weighted nonlinear Brascamp-Lieb inequalities (generalizing the results of Bennett, Carbery, Christ, and Tao \cites{bcct2008,bcct2010}) ultimately become available, it will be possible to remove the divisibility constraint at the cost of changing the definition of $\Phi_x$ to correspond to the correct weight for that context.

\section{Proof of Theorem \ref{aseq} and basic measure inequalities}
\label{gmtsec}
\subsection{Proof of Theorem \ref{aseq}}

The proof of Theorem \ref{aseq} begins with the following lemma, which generalizes Tchebyshev's inequality to finite dimensional vector spaces of functions. The heart of this generalization is to show that there exists a {\it single} set of controlled measure outside of which {\it all} functions in the vector space are uniformly bounded (when properly normalized). It extends earlier results for single-variable polynomials \cite{gressman2009} and real analytic functions \cite[Lemma 3]{gressman2017}. Although it will only be applied to Borel measures, measurability in the lemma may be taken with respect to any abstract $\sigma$-algebra.
\begin{lemma}
Suppose $\mu$ is a positive measure on some space $X$ and $\mathcal F$ is a $d$-dimensional real vector space of measurable functions from $X$ into some vector space with norm $|\cdot|$. \label{c0lemma} Then for any $\tau > 0$, there is a measurable set $E_\tau \subset X$ such that $\mu(X \setminus E_\tau) < \tau^{-1}$ for which every $f \in \mathcal F$ satisfies the inequality
\begin{equation} \sup_{x \in E_\tau} |f (x)| \leq \tau d \int |f| d \mu. \label{convex} \end{equation}
\end{lemma}
\begin{proof}
The inequality \eqref{convex} is vacuously true for any $f \in {\mathcal F}$ (regardless of $\tau$ and $E_\tau$) for which the integral on the right-hand side is infinite. It therefore suffices to prove \eqref{convex} for the subspace of those $f \in \mathcal F$ for which the integral is finite (the triangle inequality guarantees that such functions are indeed a vector space). Since this subspace also has dimension at most $d$, we may assume without loss of generality that every $f \in \mathcal F$ is $\mu$-integrable. 

Next, let $\mathcal F_0$ be the subspace consisting of all $f \in \mathcal F$ such that $\int |f| d \mu = 0$. 
If $\mathcal F_0$ is nontrivial, let $\{h_1,\ldots,h_\ell\}$ be a basis of $\mathcal F_0$ and define
\[ X_0 := \set{x \in X}{ \sum_{i=1}^\ell |h_i(x)| > 0}. \]
Because $\mathcal F_0$ is a finite-dimensional vector space (by the triangle inequality again) and because each basis element $h_i$ vanishes identically on $X \setminus X_0$, every $f \in \mathcal F_0$ is identically zero on $X \setminus X_0$. Furthermore $\mu(X_0) = 0$; this follows because
\[ \int \sum_{i=1}^\ell |h_i(x)| d \mu = 0, \]
so by the Monotone Convergence Theorem and Tchebyshev's inequality,
\[ \mu(X_0) = \lim_{N \rightarrow \infty} \mu \left( \set{x \in X}{ \sum_{i=1}^\ell |h_i(x)| > \frac{1}{N}} \right) \leq \sup_{N>0} N \! \int \sum_{i=1}^\ell |h_i(x)| d \mu = 0. \]
If $\mathcal F_0$ happens to be trivial, set $X_0 := \emptyset$.

Now let $\mathcal F_1$ be any subspace of $\mathcal F$ which has trivial intersection with $\mathcal F_0$ and satisfies $\mathcal F = \mathcal F_0  + \mathcal F_1$. If $\mathcal F_1$ is trivial, then \eqref{convex} holds because $\mathcal F = \mathcal F_0$ and consequently fixing $E_\tau := X \setminus X_0$ gives $\mu(X \setminus E_\tau) = 0$ and $\sup_{x \in E_\tau} |f(x)| = 0$ for all $f \in \mathcal F$. Thus it may be assumed that the dimension of $\mathcal F_1$ equals $d_1 \in \{1,\ldots,d\}$.
Define $S$ to be the set of all $f \in \mathcal F_1$ such that
\[ \int |f| d \mu \leq 1. \] 
The mapping $f \mapsto \int |f| d \mu$ is continuous with respect to the vector space topology, and because $\mathcal F_0 \cap \mathcal F_1$ is trivial, $f \mapsto \int |f| d \mu$ is a norm on $\mathcal F_1$, which implies that $S$ must be compact. 
Fix $\det$ to be any nonzero alternating $d_1$-linear functional on $\mathcal F_1$. By continuity and compactness, $|\det(f_1,\ldots,f_{d_1})|$ attains its maximum for some  $(f_1,\ldots,f_{d_1}) \in S^{d_1}$. Note also that the value of the maximum cannot be zero, since by scaling this would force $\det$ to be identically zero.
By Cramer's rule, for any $f \in S$,
\[ f = \sum_{j=1}^{d_1} (-1)^{j-1} \frac{\det(f, f_1,\ldots,\widehat{f_j},\ldots, f_{d_1})}{\det (f_1,\ldots,f_{d_1})} f_j \]
where the circumflex $\widehat{\cdot}$ indicates that $f_j$ is omitted from the sequence of arguments of $\det$. In particular, by the choice of the functions $f_1,\ldots,f_{d_1}$, the coefficient of each $f_j$ in this expansion of $f$ has magnitude at most one. By the triangle inequality and scaling, then, it follows that
\begin{equation} | f(x)| \leq  \left( \sum_{j=1}^{d_1} |f_j(x)| \right) \int |f| d \mu \label{pointwise} \end{equation}
for any $f \in \mathcal F_1$ and any $x \in X$. Now for any $\tau > 0$, fix
\begin{equation} E_\tau := \set{ x \in X \setminus X_0 }{ \sum_{j=1}^{d_1} |f_j(x)| \leq  \tau d }. \label{setdef} \end{equation}
By Tchebyshev's inequality,
\[ \mu(X \setminus E_\tau) < \frac{1}{\tau d} \int \left( \sum_{j=1}^{d_1} |f_j(x)| \right) d \mu(x) \leq \frac{d_1}{\tau d} \leq \frac{1}{\tau}; \]
note in particular that the first inequality is strict because
$$ \tau d  \chi_{X \setminus E_\tau} (x) < \sum_{j=1}^{d_1} |f_j(x)|$$
for each $x \in X \setminus X_0$. Equality of the integrals over $X \setminus X_0$ would force equality of the two functions $\mu$-almost everywhere on $X \setminus X_0$, which would then force $\mu(X \setminus X_0) = 0$, meaning ultimately that $\mu = 0$ and $\mathcal F_1 = \{0\}$, which has already been handled.
Taking a supremum of the inequality \eqref{pointwise} over all $x \in E_\tau$ gives
\[ \sup_{x \in E_\tau} | f(x)| \leq \tau d \int |f| d \mu \]
for any $f \in \mathcal F_1$. Since every $f \in \mathcal F$ must equal $f_0 + f_1$ for some $f_0 \in \mathcal F_0$ and $f_1 \in \mathcal F_1$ and since $f_0$ is identically zero on the given $E_\tau$, the fact that \eqref{convex} holds for $f_1$ immediately implies that it holds for $f$ as well.
\end{proof}

Before applying this lemma to the proof of Theorem \ref{aseq}, a brief remark is in order.  Although the set $E_\tau$ given by \eqref{setdef} is only described as measurable, this is generally an understatement; if the functions of $\mathcal F$ are all continuous, then $E_\tau$ is closed; if every $f \in \mathcal F$ is a polynomial, the sets $E_\tau$ are semialgebraic since they take the form
\[ \set{ x \in X}{ \sum_{j=1}^{d_1} c_j f_j(x) \leq \tau d  \mbox{ and } \sum_{i=1}^\ell \tilde c_i h_i(x) = 0 \mbox{ for all } c_j,\tilde c_i \in \{-1,1\}} \]
for functions $f_j, h_i \in \mathcal F$ which in this case are polynomials of bounded degree.

\begin{proof}[Proof of Theorem \ref{aseq}]
The proof follows rather directly from Lemma \ref{c0lemma}. Without loss of generality, it may be assumed that $\mu$ is not the zero measure on $\Omega$, since in this case $||\mathcal A||_{\mu,s} = ||\mathcal S||_{\mu,s} = \infty$. In all other cases, $||\mathcal A||_{\mu,s}$ and $||\mathcal S||_{\mu,s}$ must be finite.  First observe that 
\[ [\mu(E)]^{-k} \mathcal A(E) \leq \mathcal S(E) \]
for any measurable set $E$ with nonzero $\mu$-measure since the integrand of $\mathcal A(E)$ is pointwise dominated by $\mathcal S(E)$ on $E^k$. Consequently, for any such $E$, 
\[ ||\mathcal A||_{\mu,s} [\mu(E)]^s \leq [\mu(E)]^{-k} \mathcal A(E) \leq \mathcal S(E) \]
which then implies that $|| \mathcal A||_{\mu,s} \leq ||\mathcal S||_{\mu,s}$. To prove the remaining inequality of \eqref{aseqeq}, one applies Lemma \ref{c0lemma} with the vector space $\mathcal F$ being real-valued polynomials of degree at most $\deg \Phi$. If $m > 1$, then an arbitrary and unspecified norm $|\cdot|$ has been fixed as well; let $K^*$ be the unique symmetric, compact, convex subset of $\R^m$ such that
\begin{equation} |v| = \sup_{\ell \in K^*} |\ell \cdot v| \label{mnorm} \end{equation}
for all $v \in \R^m$, where $\cdot$ is the usual dot product.
When the inequality \eqref{convex} is applied iteratively in conjunction with Fubini's Theorem, this establishes the chain of inequalities
\begin{align*}
 \int_{E^k} & |\Phi(x_1,\ldots,x_k)| d \mu(x_1) \cdots d \mu(x_k) \\
  & \geq \int_{E^k} |(\ell \cdot \Phi)(x_1,\ldots,x_k)| \chi_{K^*}(\ell) d \mu(x_1) \cdots d \mu(x_k) \\
 & \geq \int_{E^{k-1}} (C \tau)^{-1} |(\ell \cdot \Phi_(y_1,\ldots,x_k)| \chi_{K^*}(\ell) \chi_{E_\tau}(y_1) d \mu(x_2) \cdots d \mu(x_k)   \\
 & \geq \cdots \geq (C \tau)^{-k}  |(\ell \cdot \Phi)(y_1,\ldots,y_k)| \chi_{K^*}(\ell) \chi_{E_\tau}(y_1) \cdots \chi_{E_\tau}(y_k)
 \end{align*}
 for any $y_1,\ldots,y_k \in \Omega$, 
where $C$ is the dimension of $\mathcal F$, which depends only on $n$ and $\deg \Phi$. Taking a supremum over $\ell \in \R^m$ and $y_1,\ldots,y_k$ and assuming that \eqref{snorm} holds gives that
\begin{equation} 
\begin{split}
\mathcal A(E) & \geq (C \tau)^{-k} {\mathcal S}(E_\tau) \\ & \geq (C \tau)^{-k} \left[ \mu(E_\tau) \right]^s ||\mathcal S||_{\mu,s} \geq (C \tau)^{-k} \left[ \mu(E) - \tau^{-1} \right]^s ||\mathcal S||_{\mu,s}
\end{split}  \label{t0dom} \end{equation}
for any $\tau > 1/ \mu(E)$.  If $\mu(E) \in (0,\infty)$, fixing $\tau := 2 / \mu(E)$ gives that
\[ \mathcal A(E) \geq (2C)^{-k} \left[ \mu(E) \right]^k ||\mathcal S||_{\mu,s} [ \mu(E_\lambda) ]^s \geq 2^{-s} (2C)^{-k} ||\mathcal S||_{\mu,s} [ \mu(E) ]^{k+s}. \]
If $\mu(E) = 0$ or $\mu(E) = \infty$, then the inequality immediately above still holds since it is trivial when $\mu(E) = 0$ and since the right-hand side of \eqref{t0dom} is infinite for any positive $\tau$ when $\mu(E) = \infty$.
Therefore the inequality holds for all $E$, meaning that
\[ ||\mathcal A||_{\mu,s} \geq 2^{-(s+k)} C^{-k} ||\mathcal S||_{\mu,s}, \]
which completes the main assertion \eqref{aseqeq} of Theorem \ref{aseq}. In particular, the constant depends only on $(n,k,s,\deg \Phi)$ and not on $\mu$ or the norm on $\R^m$. 
\end{proof}

\subsection{Basic GMT inequalities and Frostman's Lemma}
\label{basicgmt}

In this section, the focus returns to Theorem \ref{bestmeasthm}. The goal for the moment is to lay out some basic geometric measure theory which underlies the analytic inequality \eqref{snorm}. To that end,
given a general polynomial $\Phi : \Omega^k \rightarrow \R^m$ vanishing to order $q \geq 1$ on the diagonal as the introduction, for any $\sigma > 0$ and any $E \subset \Omega$, let
\begin{align}
\mathcal H^\sigma_\Phi(E)  & := \lim_{\delta \rightarrow 0^+} \! \inf \set{ \sum_i \left[ \mathcal S(E_i) \right]^\sigma \! }{ \chi_E \leq \sum_i \chi_{E_i}, \ \mathrm{diam}(E_i) \leq \delta}, \label{hm} \\
 \lambda^\sigma_\Phi(E) & := \lim_{\delta \rightarrow 0^+} \! \inf \left\{  \sum_{i} c_i \left[ \mathcal S(E_i) \right]^\sigma  ~ \right| \nonumber \\ & \qquad \qquad \qquad \qquad \left. \vphantom{ \left[ \mathcal S(E_i) \right]^s } \chi_E \leq \sum_i c_i \chi_{E_i}, \ c_i \geq 0, \ \mathrm{diam}(E_i) \leq \delta \right\}. \label{weightedhm}
\end{align}
To be clear, one need not assume that the sets $E_i$ have any regularity, but there is no loss of generality in requiring that each $E_i$ be Borel or even closed since continuity of $\Phi$ implies that $\mathcal S$ assigns the same value to $E_i$ and its closure $\overline{E_i}$.
The quantity $\mathcal H^{\sigma}_\Phi$ will be called the $\Phi$-Hausdorff measure of dimension $\sigma$, and as already defined in Theorem \ref{bestmeasthm}, $\lambda^{\sigma}_\Phi$ is called the weighted $\Phi$-Hausdorff measure of dimension $\sigma$. Note that $\mathcal H^{\sigma}_\Phi$ is a special case of the Carath\'{e}odory construction (see Federer \cite{federer1969} and Mattila \cite{mattila}), while $\lambda^{\sigma}_\Phi$ generalizes the measure that Howroyd \cite{howroyd1995} calls the weighted Hausdorff measure.  Just as in the definition of the classical Hausdorff measure, the quantities \eqref{hm} and \eqref{weightedhm} both define metric outer measures on $\Omega$ and therefore restrict to well-defined measures on the Borel sets; see  Folland \cite[Proposition 11.6]{folland}. 

The most basic inequalities satisfied by these quantities are that
\begin{equation} ||\mathcal S||_{\mu,\frac{1}{\sigma}}^\sigma \mu(E) \leq \lambda^{\sigma}_\Phi(E) \leq \mathcal H^{\sigma}_{\Phi}(E) \label{basicineq} \end{equation}
for any Borel set $E$ and any nonnegative Borel measure $\mu$. The first inequality follows because
\[ 
\begin{split}
|| \mathcal S||_{\mu,\frac{1}{\sigma}}^\sigma \mu(E) & = || \mathcal S||_{\mu,\frac{1}{\sigma}}^\sigma \int \chi_E d \mu \leq || \mathcal S||_{\mu,\frac{1}{\sigma}}^\sigma \int \sum_i c_i \chi_{E_i} d \mu \\
& = \sum_{i=1}^\infty c_i || \mathcal S||_{\mu,\frac{1}{\sigma}}^\sigma \mu(E_i) \leq \sum_{i=1}^\infty c_i \left[ \mathcal S(E_i) \right]^\sigma.
\end{split} \]
The latter inequality of \eqref{basicineq} follows simply because the infimum \eqref{weightedhm} is taken over a strictly larger set than \eqref{hm}. It is natural to ask when the measures $\lambda^\sigma_\Phi$ and $\mathcal H^\sigma_\Phi$ are equal or comparable. For the classical Hausdorff measure equality is known (see Federer \cite{federer1969}), but for general measures this need not be the case. In the context of this present paper, the arguments of Section \ref{mainthmsec} will establish comparability in the range $\sigma \geq n/q$ (although both measures are trivial when the inequality is strict). Beyond this observation, the question of comparability of $\lambda^\sigma_\Phi$ and $\mathcal H^{\sigma}_\Phi$ in the regime $\sigma < n/q$ will for now remain unexplored.

The measure $\lambda^\sigma_\Phi$ holds fundamental significance in the study of nonconcentration inequalities because it characterizes, via a generalization of Frostman's Lemma, the existence of nontrivial measures $\mu$ satisfying such inequalities.
\begin{lemma}
 Fix any $\sigma > 0$. There exists a nontrivial positive Borel measure $\mu$ on the compact set $K \subset \Omega \subset \R^n$ satisfying \label{frostmanL}
\begin{equation} \mathcal S(E) \geq \left[ \mu(E) \right]^\frac{1}{\sigma} \label{frostman2} \end{equation}
for all Borel sets $E \subset K$ if and only if $\lambda^{\sigma}_\Phi(K) > 0$.
\end{lemma}
\begin{proof}
The proof follows  Howroyd's proof \cite{howroyd1995} of Frostman's Lemma as given by Mattila \cite[Theorem 8.17]{mattila}.
By \eqref{basicineq}, the existence of nontrivial $\mu$ automatically guarantees that $\lambda_\Phi^\sigma(K) > 0$. Conversely, for any function $f$ on $K$, let
\[ p_{\sigma,\delta}(f) := \inf \set{ \sum_{i} c_i \left[ \mathcal S(E_i) \right]^\sigma}{ f \leq \sum_i c_i \chi_{E_i}, c_i > 0, \ \mathrm{diam}(E_i) \leq \delta}. \]
For any continuous functions $f,g$ on $K$, it is elementary to check that
\begin{align*}
p_{\sigma,\delta}(t f) & = t p_{\sigma,\delta} (f), \mbox{ for all } t \in [0,\infty), \\
p_{\sigma,\delta}(f+g) & \leq p_{\sigma,\delta}(f) + p_{\sigma,\delta}(g).
\end{align*}
It is also true that $p_{\sigma,\delta}(g) = 0$ for every nonpositive function $g$. Thus
\[ t p_{\sigma,\delta}(\chi_K) \leq p_{\sigma,\delta} (t \chi_K) \mbox{ for all } t \in \R. \]
Consequently by the Hahn-Banach Theorem, there must exist a linear functional $L$ defined on the space $C^0(K)$ of continuous functions on $K$ such that $L(\chi_K) = p_{\sigma,\delta}(\chi_K)$ and $L(f) \leq p_{\sigma,\delta}(f)$ for any continuous function $f$. If $f$ is nonnegative, $0 = - p_{\sigma,\delta}(-f) \leq L(f)$ as well, so $L$ is a positive linear functional on $C^0(K)$. By the Riesz Representation Theorem, there must be a nonnegative Borel measure $\mu_0$ on $K$ such that
\[ L(f) = \int f d \mu_0 ~ \forall f \in C^0(K) \ \mbox{ and } \ \mu_0(K) = L(\chi_K) = p_{\sigma,\delta}(\chi_K). \]
Now if $E$ is any Borel set with diameter smaller than $\delta$, let $f_j$ be a sequence of functions in $C^0(K)$ which are identically $1$ on a neighborhood of $E$, bounded above by one everywhere, and vanish outside the set $E_j$ of points distance at most $1/j$ from $E$. Then
\begin{align*}
\mu_0(E) & \leq \liminf_{j \rightarrow \infty} \int f_j d \mu_0 = \liminf_{j \rightarrow \infty} L(f_j) \leq \liminf_{j \rightarrow \infty} p_{\sigma,\delta}(f_j) \\ & \leq \liminf_{j \rightarrow \infty} [{\mathcal S} (E_j)]^\sigma = [{\mathcal S}(E) ]^\sigma,
\end{align*}
where the last inequality follows because $\Phi$ is a polynomial and therefore continuous. Finally, if $\lambda_\Phi^\sigma(K) > 0$, then there must be some positive $\delta$ such that $p_{\sigma,\delta}(\chi_K) > 0$. For this fixed value of $\delta$, $\mu_0$ must be nonzero. By subdividing $\R^n$ into nonoverlapping boxes, there must be a dyadic box $B$ of diameter less than $\delta$ such that $\mu_0(B) > 0$. Now define the measure $\mu$ by $\mu(E) := \mu_0(E \cap B)$. It follows that $\mu(\Omega) = \mu_0(B) > 0$ and for any Borel set $E \subset \Omega$ of any diameter,
\[ \mu(E) = \mu_0(E \cap B) \leq \left[ {\mathcal S}(E \cap B) \right]^\sigma \leq \left[ {\mathcal S}(E) \right]^{\sigma} \]
as desired.
\end{proof}
It is worthwhile to explicitly connect Lemma \ref{frostmanL} to D.~Oberlin's affine measure and affine curvature condition \eqref{oberlin2}.
It was observed by D.~Oberlin \cite{oberlin2003} and others that any measure $\mu$ on $\R^n$ satisfying either a nontrivial Fourier restriction inequality or $L^p$-improving convolution inequality must satisfy the inequality
\begin{equation} \mu(R) \lesssim |R|^\sigma \label{oberlin} \end{equation}
for some $\sigma > 0$ as $R$ ranges over all boxes in $\R^d$ of arbitrary orientations, i.e., all sets of points which may be expressed as products of finite intervals with respect to some orthogonal coordinates on $\R^n$. In analogy with Oberlin's affine measure\footnote{Note that Oberlin adjusts the exponent $\sigma$ so that the affine dimension of $\R^n$ is $n$, but by the present convention, the dimension is always $1$.}, let \begin{align*}
{\mathcal A}^\sigma_w(E) & := \lim_{\delta \rightarrow 0^+} \inf \left\{ \sum_{j} c_j |R_j|^\sigma \ \right| \chi_E \leq \sum_j c_j \chi_{R_j},  \\
& \qquad \qquad  \qquad \ \ \left. \vphantom{\sum_j |R_j|^{\frac{\sigma}{n}}}  \ c_j > 0, \ R_j \mbox{ are boxes of diameter} \leq \delta \right\}
\end{align*}
be called the $\sigma$-dimensional weighted affine Hausdorff measure. This weighted affine Hausdorff measure is trivially dominated by Oberlin's affine measure of dimension $n \sigma$. In this setting, Lemma \ref{frostmanL} has the following consequences:
\begin{corollary}
Suppose $K \subset \R^n$ is compact and fix any $\sigma > 0$. Then $K$ admits a nontrivial positive Borel measure $\mu$ satisfying the Oberlin affine curvature condition \eqref{oberlin}
 if and only if the $\sigma$-dimensional weighted affine Hausdorff measure of $K$ is nonzero. In particular, if ${\mathcal A}^{\sigma}_w(K) = 0$ implies that for any exponents $p_1,p_2,r_1,r_2 \in [1,\infty]$ satisfying
\[ \sigma = \frac{1}{p_1} - \frac{1}{p_2} = \frac{r_2}{r_1'} =: r_2 \left[ 1 - \frac{1}{r_1} \right], \] 
neither of the inequalities
\begin{equation*}
|| \mu * f||_{L^{p_2}(\R^n)} \lesssim ||f ||_{L^{p_1}(\R^n)} \qquad \mbox{ or } \qquad 
|| \widehat f ||_{L^{r_2} (\mu)} \lesssim ||f ||_{L^{r_1}(\R^n)}
\end{equation*}
(where $\widehat f$ denotes the Fourier transform) hold uniformly in $f$ for any nontrivial positive Borel measure $\mu$ supported on $K$.
\end{corollary}
\begin{proof}
Using Oberlin's earlier calculations \cite[Proposition 2]{oberlin2003}, it suffices to set $\Phi(x_1,\ldots,x_{n+1}) := \det (x_1 - x_{n+1},\ldots,x_{n} - x_{n+1})$ as noted in the introduction and show that the Oberlin affine curvature condition \eqref{oberlin} is equivalent to \eqref{frostman2} modulo constants and that $\mathcal A^{\sigma}_w \approx \lambda^{\sigma}_{\Phi}$. Both facts are quickly established by showing that for any bounded Borel set $E \subset \R^n$ there is a box $R$ such that $E \subset R$ and
\[ |R| \approx  \sup_{x_1,\ldots,x_{n+1} \in E} |\Phi(x_1,\ldots,x_{n+1})| \]
with implicit constants depending only on dimension. Because taking the closure of $E$ does not change the supremum, it may be assumed without loss of generality that $E$ is compact and one may fix an ensemble $x_1,\ldots,x_{n+1}$ which achieves the supremum of $|\Phi|$ on $E^{n+1}$. If the supremum is zero, then necessarily the span of all vectors $x - x_{n+1}$ as $x$ ranges over $E$ must have dimension strictly less than $n$, which implies that $E$ lies in an affine hyperplane. By boundedness of $E$, this implies that $E$ is contained in a (degenerate) box $R$ of volume zero. Otherwise the supremum is strictly positive, and by the same argument appearing in the proof of Lemma \ref{c0lemma}, it must be the case for any $x \in E$ that
\[ x = x_{n+1} + \sum_{j=1}^n c_j (x_j - x_{n+1}) \]
for constants $c_j \in [-1,1]$. The set of all such points having such an expansion is an affine image of the box $[-1,1]^n$ and consequently has Lebesgue measure $2^n |\det (x_1 - x_{n+1},\ldots,x_{n} - x_{n+1})| = 2^n \mathcal S(E)$. By the John Ellipsoid Theorem, this same set of points must be contained in an ellipsoid of comparable volume, and that ellipsoid must trivially be contained in a box $R$ of comparable volume. Thus $E \subset R$ and $|R| \lesssim \mathcal S(E)$ as promised.

Using this conclusion, if \eqref{oberlin} is assumed to hold, then for any bounded Borel set $E$,
\[ \mu(E) \leq \mu(R) \lesssim |R|^\sigma \lesssim \left[ \mathcal S(E) \right]^\sigma. \]
If $E$ is unbounded, we may write $E$ as the union of an increasing family $E_j$ of bounded Borel sets and then observe that
\[ \mu(E) = \limsup_{j \rightarrow \infty} \mu(E_j) \lesssim \limsup_{j \rightarrow \infty}  \left[ \mathcal S (E_j) \right]^\sigma \lesssim  \left[ \mathcal S(E_j) \right]^\sigma. \]
Likewise it must clearly be the case that $\lambda_\Phi^\sigma \lesssim \mathcal A_w^\sigma$ since 
$|R| \approx  \mathcal S(R)$ and since $\mathcal A_w^\sigma$ involves an infimum over a smaller class. However, for any bounded Borel sets $E_i$ such that $\sum_j c_j \chi_{E_j} \geq \chi_E$ for positive $c_j$'s, it is also true that $\sum_j c_j \chi_{R_j} \geq \chi_{E}$ for the distinguished rectangles $R_j$ containing each $E_j$. Moreover,
\[ \sum_{j} c_j  |R_j|^\sigma \lesssim \sum_j c_j \left[ \mathcal S(E_j) \right]^\sigma \]
which implies that $\mathcal A_w^\sigma \approx \lambda^\sigma_\Phi$. The corollary now follows from Lemma \ref{frostmanL}.
\end{proof}

\section{Proof of Theorem \ref{bestmeasthm}}
\label{mainthmsec}
The most difficult case of Theorem \ref{bestmeasthm} to establish is the case $\sigma = n/q$. After the cases $\sigma < n/q$ and $\sigma > n/q$ are settled (the former using Lemma \ref{frostmanL} and the latter using what amounts to a scaling argument), the proof of Theorem \ref{bestmeasthm} is reduced to the related Theorem \ref{betterthm} and ultimately to Lemma \ref{multisystem}. 

\subsection{The case $\sigma < n/q$}

The proof of Theorem \ref{bestmeasthm} in the case $\sigma < n/q$ is an almost immediate consequence of Lemma \ref{frostmanL}. First, supposing that there is a Borel measure $\mu$ satisfying \eqref{snorm} nontrivially with $s = 1/\sigma$, then $\lambda^\sigma_\Phi(\Omega) > 0$ by virtue of \eqref{basicineq} applied to the set $\Omega$ directly.

On the other hand, if $\lambda_\Phi^\sigma(\Omega) > 0$, then because $\Omega$ is an open subset of $\R^n$, it may be written as a countable increasing union of compact sets. By the Monotone Convergence Theorem, at least one of these compact subsets $K$ must have $\lambda_\Phi^\sigma(K) > 0$ as well. By Lemma \ref{frostmanL}, $K$ must admit a measure $\mu$ satisfying \eqref{snorm} nontrivially on $K$; extending $\mu$ to be zero on the complement of $K$ gives a measure $\mu$ on $\Omega$ which satisfies \eqref{snorm} nontrivially as well. In fact, it is worth noting that this argument works for any value of $\sigma$. Consequently for any $s > 0$, $\mathcal S$ admits a Borel measure satisfying \eqref{snorm} nontrivially if and only if $\lambda^{1/s}_\Phi(\Omega) > 0$. The reason for the restriction, as will be seen momentarily, is simply that $\lambda^{\sigma}_\Phi(\Omega) = 0$ if $\sigma > n/q$.

\subsection{The case $\sigma \geq n/q$: Comparison to Lebesgue measure}

The goal of this section is to establish that $\mathcal H^{\sigma}_\Phi$ must vanish when $\sigma > n/q$ and to further show when $\sigma = n/q$ that $\mathcal H^{\sigma}_\Phi$ must be absolutely continuous with respect to Lebesgue measure with an upper bound on the corresponding Radon-Nykodym derivative.
Fix standard coordinates on $\Omega \subset \R^n$. Let $\partial$ denote the $n$-tuple of partial derivatives $(\partial_{1},\ldots,\partial_{n})$ in the coordinate directions. Furthermore, for any $T \in \GL(n,\R)$, $T^* \partial$ will denote the $n$-tuple 
\[ T^* \partial := \left( \sum_{j=1}^n T_{j1} \partial_{j},\ldots, \sum_{j=1}^n T_{jn} \partial_{j} \right). \]

Assuming that $\Phi : \Omega^k \rightarrow \R^m$ is any smooth function which vanishes to order at least $q$ at every point $(x,\ldots,x) \in \Omega^k$ for every $x \in \Omega$, the main inequality to be proved in this section is that for almost every $x \in \Omega$
\begin{equation}
\begin{split} 
\frac{d \mathcal H^{\frac{n}{q}}_\Phi}{dx}& (x) \lesssim  \\  & \inf_{T \in \GL(n,\R)}  \max_{|\alpha_1| + \cdots + |\alpha_k| = q} \frac{ \left|  (T^* \partial)^{\alpha_1}_1 \cdots (T^* \partial)^{\alpha_k}_k \Phi (x,\ldots,x) \right|^{\frac{n}{q}}}{|\det T|} 
\end{split} \label{density}
\end{equation}
where $\alpha_1,\ldots,\alpha_k$ are multiindices and the subscript $j$ in $(T^* \partial)^{\alpha}_j$ indicates that the partial derivatives are applied to the argument $x_j$ of $\Phi$. The implicit constant in \eqref{density} will depend only on $k,n$, and $q$.

To begin this calculation, fix $\delta \in (0,\infty)$ and $T \in \GL(n,\R)$, and suppose that $u_1,\ldots,u_k \in [-1,1]^n$ and that $K \geq 1$ is a positive integer.  It must be the case by Taylor's Theorem that
\begin{equation}
\begin{split}
\Phi  (x' & + K^{-1} \delta T  u_1  ,\ldots, x' + K^{-1} \delta T u_k) \\ = & \  \label{ts01} K^{-q} \delta^q \sum_{|\alpha_1| + \cdots + |  \alpha_k| = q} \frac{ (T^* \partial)^{\alpha_1}_1 \cdots (T^* \partial)^{\alpha_k}_k \Phi (x',\ldots,x')}{\alpha_1! \cdots \alpha_k!} u_1^{\alpha_1} \cdots u_k^{\alpha_k} \\  & + O(K^{-q-1} \delta^{q+1})
\end{split} 
\end{equation}
for any $x'$ belonging to any fixed compact subset of $\Omega \subset \R^n$. Since $\Phi$ is smooth, the error term $O(K^{-q-1} \delta^{q+1})$ is uniform as $x$ ranges over any compact set and as $u_1,\ldots,u_k$ vary inside the box $B := [-1,1]^n$. In particular, if $x$ is any fixed point in $\Omega$ and $x' \in x + \delta T B$, then by a second application of Taylor's Theorem to the main term on the right-hand side of \eqref{ts01}, it follows that
\begin{equation}
\begin{split}
\sup_{u_1,\ldots,u_k \in B} |\Phi  (x' & + K^{-1} \delta T  u_1  ,\ldots, x' + K^{-1} \delta T u_k)| \\ \lesssim & \  \label{ts02} K^{-q} \delta^q \max_{|\alpha_1| + \cdots + |  \alpha_k| = q} |(T^* \partial)^{\alpha_1}_1 \cdots (T^* \partial)^{\alpha_k}_k \Phi (x,\ldots,x)| \\  & + O(K^{-q} \delta^{q+1}) + O(K^{-q-1} \delta^{q+1})
\end{split} 
\end{equation}
for small $\delta$ and large $K$, with an implicit constant depending only on $q,k$, and $n$, in contrast with the error terms, which may also depend on $x$, $T$, etc. From this inequality, it follows that if $C := \{C_1,\ldots,C_{K^n}\}$ is the covering of $x + \delta T B$ by the collection of $K^n$ boxes induced by subdividing $B$ into $K$ equal parts along each axis, then
\[
\begin{split}
 \sum_{i=1}^{K^n} & \sup_{y_1,\ldots,y_k \in C_i} |\Phi(y_1,\ldots,y_k)|^{\sigma} \\
 & \lesssim K^{n-q \sigma} \delta^{q \sigma} \max_{|\alpha_1| + \cdots + |  \alpha_k| = q} |(T^* \partial)^{\alpha_1}_1 \cdots (T^* \partial)^{\alpha_k}_k \Phi (x,\ldots,x)|^\frac{n}{q} \\ & \qquad + O(K^{n-q \sigma} \delta^{\sigma(q+1)}) + O(K^{n - (q+1) \sigma} \delta^{(q+1) \sigma}).
 \end{split}
  \]
As $K \rightarrow \infty$, the diameters of all sets in the covering $C$ go to zero, so taking this limit implies that
\begin{equation} \mathcal H^{\sigma}_\Phi(x + \delta T B) = 0 \mbox{ when } \sigma > \frac{n}{q} \label{largesig} \end{equation}
and that
\[ \mathcal H^{\frac{n}{q}}_\Phi(x + \delta T B) \lesssim \delta^n \! \! \! \max_{|\alpha_1| + \cdots + |  \alpha_k| = q} |(T^* \partial)^{\alpha_1}_1 \cdots (T^* \partial)^{\alpha_k}_k \Phi (x,\ldots,x)|^\frac{n}{q} + O(\delta^{\frac{n(q+1)}{q}}), \]
where just as on previous lines, the implicit constant depends only on $q,k$, and $n$.  When $\sigma > n/q$, the equality \eqref{largesig} forces $\mathcal H^{\sigma}_\Phi(\Omega) = 0$ since $\Omega$ is contained in a countable union of boxes $x + \delta T B$ with centers $x \in \Omega$. By \eqref{basicineq}, this forces $\lambda^{\sigma}_\Phi(\Omega) = 0$ as well and rules out the existence of any nontrivial Borel measure satisfying a nonconcentration inequality when $s = 1/\sigma$.

It now suffices to assume $\sigma = n/q$.
For any $x$ in a compact subset of $\Omega$ and any sufficiently small $\delta$, it has been established that
\begin{equation} 
\begin{split}
\mathcal H^{\frac{n}{q}}_\Phi & (x + \delta T B) \\
&  \label{rn0} \lesssim |x + \delta T B| \! \! \! \max_{|\alpha_1| + \cdots + |  \alpha_k| = q} \frac{|(T^* \partial)^{\alpha_1}_1 \cdots (T^* \partial)^{\alpha_k}_k \Phi (x,\ldots,x)|^\frac{n}{q}}{|\det T|} \\ & \qquad + O(\delta^{\frac{n(q+1)}{q}}) \end{split}
\end{equation}
with implicit constant depending only on $q,k$, and $n$. To reiterate: the restriction of $x$ to a compact set influences the {\it a priori} size of the error term but not the implicit constant of \eqref{rn0}. Because the maximum over $\alpha_1,\ldots,\alpha_k$ is a locally bounded function of $x$ and because $\delta^{n(q+1)/q} / |x + \delta T B| \rightarrow 0$ as $\delta \rightarrow 0^+$, it follows that for all sufficiently small $\delta$ and all $x$ in any compact set, there is a constant $C$ (depending on the compact set and the transformation $T$ as well as on $q,k$, and $n$) such that $\mathcal H^{n/q}_\Phi(x + \delta T B) \leq C |x + \delta T B|$. 
This inequality forces $\mathcal H_\Phi^{n/q}$ to be locally absolutely continuous with respect to Lebesgue measure
since any set of Lebesgue measure zero can be covered by a countable union of boxes of this form whose measures sum to any prescribed small value. Now because $\mathcal H^{n/q}_\Phi$ is known to be absolutely continuous with respect to Lebesgue measure, the Radon-Nykodym derivative can be estimated pointwise almost everywhere by dividing both sides of \eqref{rn0} by $|x + \delta T B|$ and letting $\delta \rightarrow 0^+$.
The result is that for almost every $x \in \Omega$,
\[ \frac{d \mathcal H_\Phi^{\frac{n}{q}}}{d x} \lesssim \max_{|\alpha_1| + \cdots + |  \alpha_k| = q} \frac{|(T^* \partial)^{\alpha_1}_1 \cdots (T^* \partial)^{\alpha_k}_k \Phi (x,\ldots,x)|^\frac{n}{q}}{|\det T|}. \]
 Because the inequality is true uniformly in $T$, one can take an infimum of the right-hand side over a countable dense subset of $\GL(n,\R)$ to conclude that
\begin{equation}  \frac{d \mathcal H_\Phi^{\frac{n}{q}}}{d x} \lesssim \inf_{T \in \GL(n,\R)} \max_{|\alpha_1| + \cdots + |  \alpha_k| = q} \frac{|(T^* \partial)^{\alpha_1}_1 \cdots (T^* \partial)^{\alpha_k}_k \Phi (x,\ldots,x)|^\frac{n}{q}}{|\det T|} \label{hausdens} \end{equation}
with some implicit constant depending only on $q,k$, and $n$. This is exactly the asserted inequality \eqref{density}.

It is worth observing that by homogeneity and scaling (and permuting the order of the standard coordinates), it suffices to take the infimum in $T$ over the group $\SL(n,\R)$ rather than $\GL(n,\R)$.  It should also be mentioned that since the coordinate system used to derive \eqref{density} was essentially arbitrary, one could strengthen \eqref{density} {\it a priori} even further by taking an infimum on the right-hand side over all coordinate systems. However, this apparent strengthening of \eqref{density} is not an actual improvement in this case: since all lower-order derivatives vanish, it turns out that replacing the standard coordinate partial derivatives with partial derivatives in new coordinates leaves the value of the right-hand side of \eqref{density} unchanged. This coordinate independence will be a key point in the final stages of the proof of Theorem \ref{bestmeasthm}.

\subsection{Multisystems and Theorem \ref{bestmeasthm} with $\sigma = n/q$}

The inequalities \eqref{basicineq} and \eqref{density} just proved establish that for a given $\Phi$, any measure $\mu$ satisfying \eqref{snorm} with $s = q/n$ must be absolutely continuous with respect to Lebesgue measure and must have a Radon-Nykodym derivative controlled (up to an implicit constant) by $||\mathcal S||_{q/n}^{-n/q}$ times the expression on the right-hand side of \eqref{density}.  The purpose of this section is to introduce some additional ideas which will be used to show that the upper bound given by \eqref{density} can be used to {\it define} a measure which also satisfies \eqref{snorm}.  To prove this fact, it turns out to be necessary to work with a slightly more elaborate expression and then to show that this new, more complicated expression happens to be comparable to the the right-hand side of \eqref{density}.

The added complexity which is required is to replace the standard coordinate derivatives $\partial^\alpha$  by a broader family of differential operators which includes coordinate partial derivatives in all smooth coordinates as well as some slightly more general operators. The new object under consideration will be called a multisystem.
A multisystem $\mso$ on an open set $U$ is  a collection of smooth vector fields $Y^{(i)}_j$, $i=1,\ldots,N$, $j=1,\ldots,n,$ where for each fixed $i$,  $\{Y^{(i)}_j\}_{j=1,\ldots,n}$ commute and are linearly independent at every point in $U$. The integer $N$ will be called the size of $\mso$, and the class of all multisystems of size $N$ will be denoted $\MS^{(N)}$. For any finite sequence of the form $\alpha : \{1,\ldots, a\} \rightarrow \{1,\ldots,n\}$ with $a \leq N$ and any $n$-tuple of vectors $X_1,\ldots,X_n$ at the point $p$, let
\[  (X \ms)^{\alpha} := Z^{(a)}_{\alpha_a} \cdots Z^{(1)}_{\alpha_1},  \]
where $Z^{(i)}_{\ell}$ is the unique constant-coefficient linear combination of $Y^{(i)}_1, \ldots,Y^{(i)}_n$ which equals $X_\ell$ at the point $p$. Such $\alpha$ will be called ordered multiindices in $n$ variables and $|\alpha|$ will be used to denote the order of differentiation of $(X \ms)^\alpha$, which equals the cardinality of the domain of $\alpha$. As in the previous section, $T \in \GL(n,\R)$ will also act on these differential operators by defining
\[ (T^* X)_i := \sum_{j=1}^n T_{ji} X_j \]
and taking $(T^* X \ms)^{\alpha} := ((T^* X) \ms)^{\alpha}$.

Since the remainder of this paper deals with measures on $\R^n$ which are absolutely continuous with respect to Lebesgue measure, it will be convenient to switch back and forth between analytic and geometric descriptions of these measures. In particular, every measure $\mu$ will be identified with a density $\mu(X_1,\ldots,X_n)$ which acts on $n$-tuples of vectors at the point $x$ (for $\mu$-a.e. $x \in \Omega$) by means of the correspondence
\begin{equation} \mu(X_1,\ldots,X_n) = \left| \frac{d \mu}{dx} \right| |\det (X_1,\ldots,X_n)|, \label{densitydef} \end{equation}
where the determinant is of the $n \times n$ matrix whose columns are the coefficients of the vectors $X_i$ in the standard basis. With all notation in place, it is now possible to state the main existence result for nonconcentration inequalities:
\begin{theorem}
For any $s > 0$, let $\mu$ be the density on $\Omega$ which at the point $x$ is given by \label{betterthm}
\begin{equation}
\begin{split}
 \ \mu  (X_1,&  \ldots,X_n)   := \\ &  \mathop{\inf_{\mso \in \MS^{(N)}}}_{T \in \GL(n,\R)} \max_{|\alpha_1|,\ldots,|\alpha_k| \leq N} \frac{ \left| (T^* X \ms)^{\alpha_1}_1 \cdots (T^* X \ms)^{\alpha_k}_k \Phi(x,\ldots,x) \right|^{\frac{1}{s}}}{|\det T|}. 
 \end{split} \label{betterdens}
\end{equation}
For any Borel set $E \subset \Omega$,
\begin{equation} {\mathcal S}(E) \gtrsim \left[ \mu(E) \right]^s \label{betterineq} \end{equation}
with implicit constant depending only on $(n,k,s,\deg \Phi,N)$.
\end{theorem}

It is implicit in the statement of Theorem \ref{betterthm} that the expression \eqref{betterdens} is a density in the sense of \eqref{densitydef}. To see that this is the case, it suffices to observe first that \eqref{betterdens} is zero when $X_1,\ldots,X_n$ are linearly dependent. This follows because for each $\delta > 0$, there must be a matrix $T_\delta \in \GL(n,\R)$ such that $(T^*_\delta X)_j = X_j$ for each $j$ but $\det T_\delta = \delta^{-1}$. Testing \eqref{betterdens} on this family $T_\delta$ and sending $\delta \rightarrow 0^+$ shows that the right-hand side of \eqref{betterdens} must be zero. The next step is that when $X_1,\ldots,X_n$ are linearly independent, there must be a matrix $M_X$ sending the standard basis $e_1,\ldots,e_n$ to $X_1,\ldots,X_n$, which implies that $\det M_X = \det (X_1,\ldots,X_n)$. Then because $GL(n,\R)$ is a group, one may replace $T$ everywhere on the right-hand side of \eqref{betterdens} by $(M_X^{-1})^*T$, which gives
\begin{align*}
 \mathop{\inf_{\mso \in \MS^{(N)}}}_{T \in \GL(n,\R)}&  \max_{|\alpha_1|,\ldots,|\alpha_k| \leq N} \frac{ \left| (T^* X \ms)^{\alpha_1}_1 \cdots (T^* X \ms)^{\alpha_k}_k \Phi(x,\ldots,x) \right|^{\frac{1}{s}}}{|\det T|} \\
 =   & \left[ \mathop{\inf_{\mso \in \MS^{(N)}}}_{T \in \GL(n,\R)} \max_{|\alpha_1|,\ldots,|\alpha_k| \leq N} \frac{ \left| (T^* e \ms)^{\alpha_1}_1 \cdots (T^* e \ms)^{\alpha_k}_k \Phi(x,\ldots,x) \right|^{\frac{1}{s}}}{|\det T|} \right] \\ & \qquad \cdot |\det (X_1,\ldots,X_n)|
\end{align*}
as desired.

The main lemma necessary to prove Theorem \ref{betterthm} and complete the proof of Theorem \ref{bestmeasthm} is stated below and proved in Section \ref{mainlemmasec}. It establishes the existence of a special multisystem $\mso$ and vector fields $Y_1,\ldots,Y_n$ for which it is possible to prove a kind of Bernstein or reverse Sobolev inequality on arbitrary Borel sets. 
Versions of such inequalities for intervals and boxes appear, for example, in work of Phong and Stein \cite[(2.1)]{ps1998} and Greenblatt \cite[(3.21)]{greenblatt2007}, respectively. The adaptation of such results to arbitrary Borel sets requires substantial new ideas, even in comparison to the one-dimensional version of this result appearing in \cite{gressman2009}.
 The lemma's usefulness follows from the fact that, like Lemma \ref{c0lemma}, the set $E'$ and the implicit constants are independent of the choice of $f$ within the vector space.

Assuming for the moment that Theorem \ref{betterthm} has been established, it is possible to quickly finish the proof of Theorem \ref{bestmeasthm} in the remaining special case $\sigma = n/q$. The second inequality of \eqref{twoway}, i.e.,
\[ \left[ \lambda^{\frac{n}{q}}_\Phi(E) \right]^{\frac{q}{n}} \gtrsim || \mathcal S||_{\mu,\frac{q}{n}} \left[ \mu(E) \right]^{\frac{q}{n}}, \] 
is simply a restatement of the corresponding basic inequality from \eqref{basicineq} when $\sigma = n/q$. To complete the proof of Theorem \ref{bestmeasthm}, it suffices to show when $s = q/n$ that the density \eqref{betterdens} from Theorem \ref{betterthm} is comparable to or greater than the density on the right-hand side of \eqref{hausdens} which dominates $d \mathcal H^{n/q}_\Phi / dx$. Once this is known, if $\mu$ is the measure promised by Theorem \ref{betterthm} when $s = q/n$, 
\[ \left[ \mathcal H^{n/q}_\Phi(E) \right]^{\frac{q}{n}} \lesssim  \left[ \mu(E) \right]^\frac{q}{n} \lesssim \mathcal S(E) \]
for any Borel set $E$, with uniform implicit constants depending only on the parameters $(q,k,n,\deg \Phi)$, because $\mu$ dominates $\mathcal H^{n/q}_\Phi$ by comparison of densities and $\mu$ satisfies \eqref{snorm} by Theorem \ref{betterthm}. Combining with the basic inequalities \eqref{basicineq} gives
\[ \mu(E) \approx \lambda^{\frac{n}{q}}_\Phi(E) \approx \mathcal H^{\frac{n}{q}}_\Phi(E) \]
for all Borel sets $E$, with implicit constants depending only on $(q,k,n,\deg \Phi)$. To reiterate,  $\mu$ is dominated by $\mathcal H^{n/q}_\Phi$ by virtue of the basic inequalities \eqref{basicineq}, so the densities from \eqref{betterdens} and \eqref{hausdens} must in fact be comparable, and thus the upper bound \eqref{hausdens} improves to become
\begin{equation} \frac{d \lambda_\Phi^{\frac{n}{q}}}{dx} \approx  \frac{d \mathcal H_\Phi^{\frac{n}{q}}}{d x} \approx \inf_{T \in \GL(n,\R)} \mathop{\max_{|\alpha_1| + \cdots }}_{+ |  \alpha_k| = q} \frac{|(T^* \partial)^{\alpha_1}_1 \cdots (T^* \partial)^{\alpha_k}_k \Phi (x,\ldots,x)|^\frac{n}{q}}{|\det T|} \label{hausdens2} \end{equation} 
with implicit constants depending only on $(k,n,q,\deg \Phi)$.

Thus, assuming Theorem \ref{betterthm} it suffices to compare the densities from \eqref{density} and \eqref{betterdens}, and show that the latter dominates the former.  In so doing, it further suffices to fix $X_1,\ldots,X_n$ to be the standard coordinate vectors on $\Omega \subset \R^n$. Now because $\Phi$ vanishes to order $q$ on $\Delta$, it must be the case that
\[ (T^* X \ms)^{\alpha_1}_1 \cdots (T^* X \ms)^{\alpha_k}_k \Phi(x,\ldots,x) = (T^* \partial)^{\alpha_1}_1 \cdots (T^* \partial)^{\alpha_k}_k \Phi (x,\ldots,x) \]
whenever $|\alpha_1| + \cdots + |\alpha_k| = q$ since the two differential operators have equal highest-order parts and the lower-order terms are all differential operators of order $q-1$ and lower (Note that for any ordered multiindex $\alpha_j$, the operator $\partial^\alpha_j$ makes sense as a standard multiindex because the coordinate vector fields commute.)  Therefore the inequality
\begin{align*}
\inf_{T \in \GL(n,\R)} &  \max_{|\alpha_1| + \cdots + |\alpha_k| = q} \frac{ \left|  (T^* \partial)^{\alpha_1}_1 \cdots (T^* \partial)^{\alpha_k}_k \Phi (x,\ldots,x) \right|^{\frac{n}{q}}}{|\det T|} \\
& \leq \mathop{\inf_{\mso \in \MS^{(N)}}}_{T \in \GL(n,\R)}  \max_{|\alpha_1|,\ldots,|\alpha_k| \leq q} \frac{ \left| (T^* X \ms)^{\alpha_1}_1 \cdots (T^* X \ms)^{\alpha_k}_k \Phi(x,\ldots,x) \right|^{\frac{n}{q}}}{|\det T|}
 \end{align*}
must hold. Thus the final portions of Theorem \ref{bestmeasthm} will follow once the proof of Theorem \ref{betterthm} is complete.

Theorem \ref{betterthm} is itself a rather direct consequence of the following lemma:
\begin{lemma}
Suppose that $\mu$ is a nonnegative Borel measure on $\Omega \subset \R^n$ which is absolutely continuous with respect to Lebesgue measure with locally integrable Radon-Nykodym derivative.  Let $d \geq 1$ and $N \geq 1$ be fixed positive integers.
Given any bounded Borel set $E \subset \Omega$ of finite, nonzero $\mu$-measure, there exists an open set $U$, a multisystem $\mso$ of size $N$ on $U$, vector fields \label{multisystem} $Y_1,\ldots,Y_n$ on $U$, and a Borel set $E' \subset U \cap E$ such that
\begin{enumerate}
\item $\mu(E') \gtrsim \mu(E)$
\item $\mu(Y_1,\ldots,Y_n) \gtrsim \mu(E)$ at every point of $E'$.
\item For every polynomial map $f : \Omega \rightarrow \R^m$  of degree at most $d$ and every ordered multiindex $\alpha$ with $|\alpha| \leq N$,
\begin{equation} \sup_{x \in E'} |(Y \ms)^\alpha f(x)| \lesssim \sup_{x \in E} |f(x)|. \label{revsob} \end{equation}
\end{enumerate}
The implicit constants depend only on $(n,d,N)$.
\end{lemma}
\begin{proof}[Proof of Theorem \ref{betterthm} assuming Lemma \ref{multisystem}.]
At this point, the proof of Theorem \ref{betterthm} is almost the same as the proof of Theorem \ref{aseq}. Let $E$ be a bounded Borel measurable set with positive $\mu$ measure. Fix an integer $N > 0$ and let the multisystem $\mso$, vector fields  $Y_1,\ldots,Y_n$, and sets $E'$ and $U$ be as in Lemma \ref{multisystem}. Let $y$ be any point in $E'$.  If $\alpha_1,\ldots,\alpha_k$ are ordered multiindices such that $|\alpha_i| \leq N$ for all $i=1,\ldots,k$, then
\begin{align*}
\sup_{(x_1,\ldots,x_k) \in E^{k}} |\Phi(x_1,\ldots,x_k)| & \gtrsim
\sup_{(x_1,\ldots,x_{k-1}) \in E^{k-1}} |(Y \ms)_k^{\alpha_k} \Phi(x_1,\ldots,x_{k-1},y)| \\
& \gtrsim \cdots \gtrsim |(Y \ms)^{\alpha_1}_1 \cdots (Y \ms)^{\alpha_k}_k \Phi(y,\ldots,y)|.
\end{align*}
Taking a maximum over $\alpha_1,\ldots,\alpha_k$ and comparing to the definition \eqref{betterdens} of the density $\mu$ (fixing $T$ to be the identity), it follows that
\[ \sup_{(x_1,\ldots,x_k) \in E^{k}} |\Phi(x_1,\ldots,x_k)| \gtrsim \left[ \left. \mu(Y_1,\ldots,Y_n) \right|_y  \right]^s \gtrsim \left[ \mu(E) \right]^s. \]
This is exactly the desired inequality \eqref{betterineq}. If $\mu(E) = 0$, the inequality \eqref{betterineq} is trivial, so the only remaining case is when $E$ is an unbounded Borel set. In this case, $E = \bigcup_{M=1}^\infty E_M$, where $E_M := E \cap \set{x \in \Omega}{|x| \leq M}$. Then by Monotone Convergence,
\[ {\mathcal S}(E) \geq \sup_{M} {\mathcal S}(E_M) \gtrsim \sup_{M} \left[ \mu(E_M) \right]^s = \left[ \mu(E) \right]^s \]
as desired.
\end{proof}

\subsection{Remarks on calculation}
\label{calcsec}
Before proceeding with the proof of Lemma \ref{multisystem}, it is perhaps worthwhile to make some elementary remarks regarding the infimum appearing in \eqref{betterdens} or \eqref{hausdens} since from a practical perspective it represents the most difficult part of any actual calculation of the density. 
If $\alpha$ is any ordered multiindex of order $d$, then by multilinearity it follows for any invertible square matrices $T$ and $O$ that 
\[ ((TO^{-1})^* X \ms)^\alpha = \sum_{|\beta| = d} O_{\beta_1 \alpha_1}^{-1} \cdots O_{\beta_d \alpha_d}^{-1} (T X \ms)^\beta. \]
If, for example, $O$ is an orthogonal matrix, it must then be the case that
\begin{equation}
\begin{split}
  \max_{|\alpha_1|,\ldots,|\alpha_k| \leq N} &  |((TO^{-1})^* X \ms)^{\alpha_1}_1  \cdots ((T O^{-1})^* X \ms)^{\alpha_k}_k \Phi(x,\ldots,x) | \\ 
  &\leq n^{Nk} \max_{|\alpha_1|,\ldots,|\alpha_k| \leq N}  |(T^* X \ms)^{\alpha_1}_1 \cdots (T^* X \ms)^{\alpha_k}_k \Phi(x,\ldots,x) | \label{demobound}
  \end{split}
  \end{equation}
by simply using the fact that $|O_{jk}^{-1}| \leq 1$ and making the conservative estimate that the number of terms in the expanded multilinear sum is never greater than $n^{Nk}$. This simple calculation shows that the infimum over $T \in \GL(n,\R)$ in \eqref{betterdens} is always comparable (up to a factor depending only on $n,k,N,$ and $s$) to the infimum over all matrices in some fixed subset $\mathcal G \subset \GL(n,\R)$  provided that every matrix $T \in \GL(n,\R)$ has a factorization $T = G O$ where $G \in \mathcal G$  and $O$ is orthogonal. The propositions below demonstrate two slightly different applications of this same idea.

The first example is based on the Singular Value Decomposition. Using this simplification, it is possible to characterize the positivity of the density \eqref{hausdens2} pointwise in terms of a height-type criterion for certain Newton-like polytopes. Algebraically, the proposition is closely related to the Hilbert-Mumford criterion, which was first proved in the real-valued case proved by Birkes \cite{birkes1971}.
\begin{proposition}
For any $x \in \Omega$, if $\Phi$ vanishes to order $q$ at $(x,\ldots,x)$, then
\[ \inf_{T \in GL(n,\R)} \mathop{\max_{|\alpha_1| + \cdots}}_{+ |\alpha_k| = q} \frac{ \left| (T^* \partial)_1^{\alpha_1} \cdots (T^* \partial)_k^{\alpha_k} \Phi(x,\ldots,x) \right|^{\frac{n}{q}}}{|\det T|} > 0 \] if and only if
for every orthogonal matrix $O$, the point $(q/n,\ldots,q/n) \in [0,\infty)^n$ belongs to the convex hull in $[0,\infty)^{n}$ of the set
\begin{equation} \set{ \alpha_1 + \cdots + \alpha_k }{ (O^* \partial)^{\alpha_1}_1 \cdots (O^* \partial)_k^{\alpha_k} \Phi(x,\ldots,x) \neq 0, \ \sum_{j=1}^k |\alpha_j| = q}. \label{theset} \end{equation}
\end{proposition}
\begin{proof}
By the SVD, every $T \in GL(n,\R)$ factors as $T = O_1 D O_2$ where $O_1,O_2 \in O(n,\R)$ and $D$ is a nonnegative diagonal matrix. If the diagonal entries of $D$ are denoted $(t_1,\ldots,t_n)$, the expansion analogous to \eqref{demobound} gives that
\begin{equation}
\begin{split}
 n^{qk} & \inf_{T \in GL(n,\R)}   \mathop{\max_{|\alpha_1| + \cdots}}_{+ |\alpha_k| = q} \frac{ \left| (T^* \partial)_1^{\alpha_1} \cdots (T^* \partial)_k^{\alpha_k} \Phi(x,\ldots,x) \right|^{\frac{n}{q}}}{|\det T|} \\
 \geq
 & \mathop{\inf_{O_1 \in O(n,\R)}}_{t \in (0,\infty)^n} \mathop{\max_{|\alpha_1| + \cdots}}_{+ |\alpha_k| = q} t^{- {\bf 1} + \frac{n}{q} \sum_{j=1}^k \alpha_j} \left| (O_1^* \partial)_1^{\alpha_1} \cdots (O_1^* \partial)_k^{\alpha_k} \Phi(x,\ldots,x) \right|^{\frac{n}{q}} \label{hmc}
 \end{split}
 \end{equation}
 where ${\bf 1} := (1,\ldots,1) \in \Z^n$. It is also trivially true that the inequality \eqref{hmc} is reversed when the factor of $n^{qk}$ is omitted. Thus it suffices to find necessary and sufficient conditions for the quantity on the right-hand side of \eqref{hmc} to be nonzero. For convenience, let $a$ denote any $k$-tuple of multiindices $(\alpha_1,\ldots,\alpha_k)$ with $|\alpha_1| + \cdots + |\alpha_k| = q$, and define $\Sigma a := \alpha_1 + \cdots + \alpha_k$ and
 \[ C_a :=   \left| (O_1^* \partial)_1^{\alpha_1} \cdots (O_1^* \partial)_k^{\alpha_k} \Phi(x,\ldots,x) \right|^{\frac{n}{q}}. \] 
 If $(q/n) {\bf 1}$ belongs to the convex hull of the set \eqref{theset} for every $O$, then for every $O$ it must be possible to find $a_1,\ldots,a_{N_O}$ and $\theta_1,\ldots,\theta_{N_O} \in [0,1]$ such that $\theta_1 + \cdots + \theta_{N_O} = 1$,
 \[ \sum_{j=1}^{N_O} \theta_j \Sigma a_j = \frac{q}{n} {\bf 1}, \]
 and $C_{a_j} > 0$ for $j=1,\ldots,N_O$. Because a maximum of terms always dominates any convex combination, it follows that
 \begin{equation} \inf_{t \in (0,\infty)^n} \max_{a} t^{-{\bf 1} + \frac{n}{q} \Sigma a} C_a \geq \inf_{t \in (0,1)^n} \prod_{j=1}^{N_O} (t^{-{\bf 1} + \frac{n}{q} \Sigma a_j} C_{a_j})^{\theta_j}  = \prod_{j=1}^{N_O} (C_{a_j})^{\theta_j}. \label{info} \end{equation}
The quantities $C_a$ are continuous functions of $O$ and nonzero at the particular $O$ in question, so each $C_{a_j}$ is strictly positive on a neighborhood of $O$ and consequently the infimum \eqref{info} must be bounded below by a positive quantity on a neighborhood of $O \in O(n,\R)$. By compactness of the orthogonal group, the infimum \eqref{hmc} must be strictly positive.

If, on the other hand, there is some $O \in O(n,\R)$ such that $(q/n) \bf 1$ does not belong to the convex hull of \eqref{theset}, then the Separating Hyperplane Theorem guarantees the existence of $\ell \in \R^n$ such that $\ell \cdot \Sigma a > (q/n) \ell \cdot \bf 1$ for all $a$. Taking $t = (e^{-s \ell_1},\ldots,e^{-s \ell n})$ gives
\[ t^{-{\bf 1} + \frac{n}{q} \Sigma a} = e^{- \frac{s n}{q} \ell \cdot \left( \Sigma a - \frac{q}{n} {\bf 1} \right)} \rightarrow 0 \]
as $s \rightarrow \infty$ for all $a$. Consequently the infimum \eqref{hmc} must be zero.
\end{proof}

For the second example, recall the determinantal Hausdorff measure from Section \ref{examples1}. In that section, it was claimed that
\[ |E| \lesssim \sup_{A_1,A_2 \in E} |\det (A_1 - A_2)|^n \]
for any Borel set $E \subset \R^{n \times n}$. By virtue of Theorem \ref{bestmeasthm}, to prove this inequality, it suffices to show that the density \eqref{hausdens2} is uniformly bounded below. This calculation is relatively straightforward for triangular matrices $T$ and is recorded in the following proposition:
\begin{proposition}
Let \[ \Phi(A_1,A_2) = \det (A_1 - A_2), \] 
where $A_1$ and $A_2$ denote matrices in $\R^{n \times n}$. Then the Radon-Nykodym derivative  $d \lambda^{n}_\Phi / dx$ is uniformly bounded below by a constant depending only on $n$.  \label{calcprop}
\end{proposition}
\begin{proof}
Before beginning, note that the correct $\Phi$-Hausdorff dimension for this problem is $n$ because $n^2$ is the dimension of the parameter space $\R^{n \times n}$ and $q = n$ is the order of vanishing of $\Phi$ on the diagonal.

Order the entries $(i,j)$ of $n \times n$ matrices lexicographically and let $\partial_{ij}$ correspond to differentiation in the direction of the $(i,j)$ entry. For any $T \in GL(n \times n,\R)$, one may write $T = L Q$ for a lower triangular matrix $L$ and an orthogonal matrix $Q$ (this is just the so-called $QR$ decomposition applied to $T^*$). Consequently, in taking the infimum \eqref{hausdens2}, up to a uniform constant, it suffices to assume that $T$ is lower triangular; in this case the directional derivatives $Y_{ij} := (T^* \partial)_{ij}$ are spanned by $\partial_{i'j'}$ for those entries $(i',j')$ which are lexicographically greater than or equal to $(i,j)$.

Because the determinant is a linear function of each column and each row of a matrix,
\[ \partial_{i_1 j_1} \cdots \partial_{i_n j_n} \det (\cdot) = 0 \]
if either the indices $i_1,\ldots,i_n$ or the indices $j_1,\ldots,j_n$ are not distinct. When both the $i$'s and the $j$'s are distinct, the value of the derivative is $\pm 1$ depending on the relative orderings of the indices.
By definition of the directional derivatives $Y_{ij}$, the differential operator $Y_{1\ell_1} \cdots Y_{n \ell_n}$ can always be written as a linear combination of derivatives $\partial_{i_1 j_1} \cdots \partial_{i_n j_n}$ where $(i_1,j_1) \geq (1,\ell_1),\ldots,(i_n, j_n) \geq (n , \ell_n)$ lexicographically. However, among all such possible choices of the entries $(i_1,j_1),\ldots,(i_n,j_n)$, there is only one possibility where the $i$'s and $j's$ are distinct: $(i_1,j_1) = (1,\ell_1),\ldots,(i_n,j_n) = (n,\ell_n)$. This is because $i_1 \geq 1,\ldots,i_n \geq n$, so by the Pigeonhole Principle, the $i$'s can only be distinct when $i_1 = 1,\ldots i_n = n$. This forces $j_1 \geq \ell_1,\ldots,j_n \geq \ell_n$, which implies $j_1 = \ell_1,\ldots,j_n = \ell_n$ for the same reason because $\ell_1,\ldots,\ell_n$ are already distinct. Therefore
\[ Y_{1\ell_1} \cdots Y_{n \ell_n} \det (\cdot) = \pm c_{1 \ell_1} \cdots c_{n \ell_n} \]
where $c_{ij}$ is the coefficient of $\partial_{ij}$ in the expansion of $Y_{ij}$. It follows that
\[ \left| \prod_{\sigma \in \mathfrak S_n} \left[ Y_{1 \sigma_1} \cdots Y_{n \sigma_n} \det ( \cdot) \right] \right|^{\frac{1}{n!}} = \prod_{i,j=1}^n |c_{ij}|^{\frac{1}{n}} \]
since each entry $(i,j)$ appears in a $1/n$ fraction of all permutations $\sigma$. Because $T$ is lower triangular, the product of all $|c_{ij}|$ is just the absolute value of the determinant. Therefore
\[ \max_{|\alpha| = n} |(T^* \partial)^{\alpha}_1 \Phi(A,A)| \geq |\det T|^\frac{1}{n}\]
for any lower triangular matrix $T$.
Raising both sides to the power $n$ gives exactly the desired lower bound for the density \eqref{hausdens2}.
\end{proof}

As a final remark on calculation, note that the simplifications used above apply equally well to Theorem \ref{betterthm}. Using the QR decomposition as above, for example, it is possible to show that the function $\Phi$ on $\R^2 \times \R^2$ given by
\[ \Phi((x_1,y_1),(x_2,y_2)) = (x_1 - x_2)^2 + (y_1 - y_2)^3 \]
satisfies the nonconcentration inequality
\[ \mathcal S(E) \gtrsim |E|^{\frac{6}{5}}, \]
which is an interesting result because this $\Phi$ is degenerate when $\sigma = n/q = 2/2$. The necessary calculation is relatively simple when one assumes without loss of generality that one of the two vectors in the pair $T^* X$ points in the $y$-direction.

\section{Proof of Lemma \ref{multisystem}}
\label{mainlemmasec}
\subsection{Construction of the multisystem}
\begin{proof}
The proof begins by establishing that it suffices to assume that the functions $f$ are scalar-valued, i.e., that $m=1$.  When $m > 1$, as previously noted in \eqref{mnorm}, there must exist a symmetric, compact, convex set $K^* \subset \R^m$ such that
\[ |v| = \sup_{\ell \in K^*} | \ell \cdot v | \]
for all $v \in \R^m$.
Taking $f := (f_1,\ldots,f_m)$ to be a polynomial map of degree $d$ and assuming the lemma for the case $m=1$ gives
\begin{align*}
\sup_{x \in E'} |(Y \ms)^\alpha f(x)| &  =   \sup_{x \in E'} \sup_{\ell \in K^*} |(Y \ms)^\alpha (\ell \cdot f)(x)| \\
& \lesssim \sup_{\ell \in K^*} \sup_{x \in E} |(\ell \cdot f)(x)| = \sup_{x \in E} |f(x)|,
\end{align*}
so the implicit constant can taken to be independent of $m$ and of the choice of norm $|\cdot|$ on $\R^m$.

Let $\mathcal F_0$ be the vector space of polynomials $f$ of degree at most $d$ and let $D := \dim \mathcal F_0 $. Because $E$ is bounded, all polynomials of degree $d$ are bounded on $E$, and because $E$ has nonzero $\mu$ measure, no nontrivial polynomial can vanish identically on $E$. Thus $f \mapsto \sup_{x \in E} |f(x)|$ is a norm on $\mathcal F_0$, and
as in the proof of Lemma \ref{c0lemma}, one may fix $\det$ to be any nonzero alternating $D$-linear form on $\mathcal F_0$. Using this $\det$ just as was done earlier, it is possible to find $f_1,\ldots,f_D \in {\mathcal F}_0$ such that $\sup_{x \in E} |f_j(x)| \leq 1$ 
and
\begin{equation} f = \sum_{j=1}^N c_j f_j \label{coeffs} \end{equation}
for any $f \in \mathcal F_0$
with constants $c_j$ satisfying $|c_j| \leq \sup_{x \in E} |f(x)|$ for each $j=1,\ldots,D$. For any $n$-tuple $(j_1, j_2,\ldots,j_n)$ of indices in $\{1,\ldots,D\}$ such that $j_1 < j_2 < \cdots < j_n$, let $U_{j_1,\ldots,j_n}$ be the open set of points $x \in \Omega$ such that
\[ \left|  \left. d f_{j_1} \wedge \cdots \wedge d f_{j_n} \right|_x \right| > \frac{1}{2} \left| \left. d f_{i_1} \wedge \cdots \wedge df_{i_n} \right|_x \right| \mbox{ for all } i_1,\ldots,i_n \in \{1,\ldots,D\} \]
where $\left. d f \right|_x$ denotes the exterior derivative of $f$ at the point $x$. The union of all $U_{j_1,\ldots,j_n}$ over all possible $j_1 < \cdots < j_n$ must be all of $\Omega$ because at every point $x$ there must be some $j_1 < \cdots < j_n$ for which $\left. d f_{j_1} \wedge \cdots \wedge d f_{j_n} \right|_x$ is nonzero. Since these open sets cover $\Omega$, they cover $E$ as well, and there must consequently be a single choice of $j_1 < \cdots < j_n$ such that $\mu(E \cap U_{j_1,\ldots,j_n}) \geq D^{-n} \mu(E)$. On $U := U_{j_1,\dots,j_n}$, define vector fields $Y_1,\ldots,Y_n$ by means of the formula
\begin{equation} Y_i f := \frac{d f_{j_1} \wedge \cdots \wedge d f \wedge \cdots \wedge d f_{j_n}}{d f_{j_1} \wedge \cdots \wedge d f_{j_n}}, \label{vecdef} \end{equation}
where $df$ in the numerator appears in position $i$ of the wedge product and replaces $df_{j_i}$. This means that $Y_i f_{j_{i'}}$ vanishes if $i \neq i'$ and is identically one on $U_{j_1,\ldots,j_n}$ if $i = i'$, which further means that the $Y_i$ are locally coordinate vector fields and commute with one another. Moreover, by \eqref{coeffs} and the definition of $U_{j_1,\ldots,j_n}$, it must be the case that
\[ |Y_i f| \leq \sum_{j=1}^D | Y_i f_j| \sup_{x \in E} |f(x)| \leq 2D \sup_{x \in E} |f(x)|\]
at every point of $U_{j_1,\ldots,j_n}$. Furthermore
\begin{align*} 
\int_{E \cap U}  & | \mu(Y_1,\ldots,Y_n)|^{-1} d \mu \\
& =  \int_{E \cap U} | \mu(Y_1,\ldots,Y_n)|^{-1} \frac{d \mu}{|d f_{j_1} \wedge \cdots \wedge d f_{j_n}|} |d f_{j_1} \wedge \cdots \wedge d f_{j_n}| \\
& = \int_{E \cap U} \frac{1}{|(d f_{j_1} \wedge \cdots \wedge d f_{j_n})(Y_1,\ldots,Y_n)|} |d f_{j_1} \wedge \cdots \wedge d f_{j_n}| \\
& = \int_{E \cap U} |d f_{j_1} \wedge \cdots \wedge d f_{j_n}|,
\end{align*}
and by the change of variables formula, the last integral will be bounded above by the maximum number of nondegenerate solutions (i.e., solutions where the Jacobian determinant of the system is nonzero) of the system of equations 
\begin{equation}
f_{j_1}(x) = a_1,\ldots,f_{j_n}(x) = a_n \label{systems}
\end{equation} in $E \cap U$ for $a_1,\ldots,a_n \in [-1,1]$ since $|f_{j_i}(x)| \leq 1$ on $E$. Letting $S$ denote a uniform upper bound for this number of solutions, it follows from Tchebyshev's inequality that there is a measurable set $E' \subset E \cap U$ with $\mu(E') \geq \frac{1}{2} D^{-n} \mu(E)$ such that
\[ \mu(Y_1,\ldots,Y_n) \geq \frac{1}{2} D^{-n} S^{-1} \mu(E). \]
This completes the proof of Lemma \ref{multisystem} in the case $N=1$.

By induction, assume the lemma has been established up to some level $N-1$. For convenience, let the sets $E'$ and $U$ at stage $N-1$ be denoted $E_{N-1}$ and $U_{N-1}$, respectively. Suppose also that the lemma has been proved for some class of functions $\mathcal F_{N-1}$ which includes all polynomials of degree $d$. Stage $N$ follows by applying the already-established base case of the lemma to the space of functions $\mathcal F_{N}$ on $U_{N-1}$ which defined to be the span of $\mathcal F_{N-1}$ and $Y_i \mathcal F_{N-1}$, $i=1,\ldots,n$.  Postponing for the moment the problem of counting solutions of systems of equations during this induction procedure, it must be the case that for any $N$, there is an open set $U_N$ and some measurable $E_N \subset E$ such that $\mu(E_N \cap U_N) \gtrsim \mu(E)$ for some implicit constant depending on $(n,d,N)$ and there is a multisystem $\mso$ of size $N$, formed by extending the multisystem $\mso$ of size $N-1$ to add new vector fields $Y^{(N)}_j := Y_j$ defined by \eqref{vecdef} on $U_{N}$ as above.
For this extended multisystem, it must be the case that
\begin{equation} |Y^{(N)}_{j_N} \cdots Y^{(1)}_{j_1} f| \leq C_N \sup_{x \in E} |f(x)| \label{indhyp} \end{equation}
for all $j_1,\ldots,j_N$ and all $f \in \mathcal F_0$. Moreover, because each collection $Y^{(i)}_{1},\ldots,Y^{(i)}_{n}$ is locally given by coordinate vector fields with local coordinate functions which themselves belong to the finite-dimensional function space $\mathcal F_{N-1}$, it follows that
\[ Y^{(i+1)}_j = \sum_{\ell=1}^n (Y^{(i+1)}_j f_\ell) Y^{(i)}_\ell \]
when $f_1,\ldots,f_n$ are the functions used to construct the $Y^{(i)}_\ell$. In particular, the coefficients $|Y^{(i+1)}_j f_\ell|$ are bounded uniformly in $j$ and $\ell$ (and uniformly in $E$ and $\mu$). By induction, this implies that the final vector fields $Y^{(N)}_{j}$ are linear combinations of the $Y^{(i)}_\ell$ for $i < N$ with coefficients that are uniformly bounded. 
Because the vectors $Y^{(N)}_j$ may be written as linear combinations of all previous $Y^{(i)}_{\ell}$ with bounded linear coefficients, it follows from \eqref{indhyp} that
\[ \sup_{x \in E_N} |(Y^{(N)} \ms)^{\alpha} f(x)| \lesssim \sup_{x \in E} |f(x)| \]
with implicit constant independent of $\mu$ and $E$ whenever $\alpha$ is an ordered multiindex with $|\alpha| \leq N$.  Taking the vector fields $Y^{(N)}_1,\ldots,Y^{(N)}_n$ to be vector fields promised in the statement of the lemma together with $E' := E_N$ and $U := U_N$ completes the proof with the exception of the unfinished business of counting solutions of systems of equations.
\end{proof}

\subsection{Underlying geometry and solution counting}

The problem of counting solutions is an independent algebraic issue which has already been addressed elsewhere in the case of real analytic functions \cite{gressman2017}, so the reader who is not interested in the precise nature of the implicit constants in Theorem \ref{bestmeasthm} may skip the rest of this section and consider Theorem \ref{bestmeasthm} fully proved. For those who continue reading, there are two main purposes to this section. The first is to establish that the systems of equations encountered in the previous section have a bounded number of isolated solutions with an upper bound depending only on the constants $(n,d,N)$ as promised. The second major purpose of this section is to demonstrate that there is an intrinsic geometric object which governs the possible number of solutions. This means that a finite upper bound will continue to hold uniformly even when the functions $f$ belong, for example, to some o-minimal structure. This intrinsic geometric object is also closely related to certain geometric differential operators which were constructed some time ago to study uniform coordinate-independent sublevel set estimates \cite{gressman2010II}. In a very precise way, the object described below allows one to extend those earlier differential operators to a broader class which includes rational functions of the simpler objects.

Throughout this section, the open set $\Omega \subset \R^n$ and the polynomials of bounded degree on $\Omega$ will be regarded as simply an abstract smooth manifold $\mathcal M$ of dimension $n$ and a finite-dimensional vector space $\mathcal F$ of smooth functions on $\mathcal M$. Given such a pair $(\mathcal M, \mathcal F)$, a new pair $(\mathcal M', \mathcal F')$, representing a sort of abstract derivative of the original pair, is constructed as follows. 
 Let $\mathcal M'$ be the bundle $\Lambda^n_*(\mathcal M)$ of nonvanishing $n$-forms over points of $\mathcal M$, i.e., points of $\mathcal M'$ are nonvanishing $n$-forms $\omega_x$, where the subscript $x$ is used to indicate that $\omega_x$ acts as an alternating $n$-linear form on the tangent space at $x \in \mathcal M$.  Let $\mathcal F'$ be the vector space of smooth functions on $\mathcal M'$ spanned by the functions
\[ f(\omega_x) :=  f (x), \qquad f \in \mathcal F, \]
and
\[  \frac{\left. \left( df_{1} \wedge \cdots \wedge d f_{n} \right) \right|_{x}}{\omega_x}, \qquad f_1,\ldots,f_n \in \mathcal F. \]
The construction of $(\mathcal M', \mathcal F')$ allows one to extend the class of functions $\mathcal F$ to a broader class involving derivatives of the functions in $\mathcal F$ without constructing vector fields or coordinate systems. The cost of the construction is the change of dimension of $\mathcal M$ from $n$ to $n+1$, which roughly corresponds to including a new indeterminate variable. If $\mathcal M$ is the one-dimensional interval $(a,b)$, for example, then one can show that $\mathcal M$ is diffeomorphic to $(a,b) \times \R_{\neq 0}$ and $\mathcal F'$ is spanned by the functions $f(t)$ for $f \in \mathcal F$ and functions of the form $s f'(t)$ where $s \neq 0$ is the new indeterminate. In higher dimensions, the situation is somewhat more complex but still analogous.

Iterating the construction of $\mathcal M'$ and $\mathcal F'$ gives a sequence of manifolds $\mathcal M^{(i)}$ and function spaces $\mathcal F^{(i)}$ on $\mathcal M^{(i)}$, $i=0,\ldots,N$ (with $\mathcal M^{(0)} := \mathcal M$  and $\mathcal F^{(0)} := \mathcal F$). The spaces $\mathcal M^{(i)}$ have dimension $n + i$ and have fiber bundle projections $p_i$
\[ \mathcal M^{(i)} \stackrel{p_i}{\rightarrow} \mathcal M^{(i-1)} \stackrel{p_{i-1}}{\rightarrow} \cdots \stackrel{p_1} {\rightarrow} \mathcal M^{(0)}. \]
For convenience, let $\pi^{(i)}$ be the projection map $p_1 \circ \cdots \circ p_i$ from $\mathcal M^{(i)}$ to $\mathcal M^{(0)}$.
The space $\mathcal F^{(i)}$ is spanned by functions of the forms
\[ f (\omega_x) := (f \circ p_i) (\omega_x), \qquad f \in \mathcal F^{(i-1)} \]
and
\begin{equation}  
\left. d^{n+i-1}(f_1,\ldots,f_{n+i-1}) \right|_{\omega_x} := \frac{\left. \left( df_{1} \wedge \cdots \wedge d f_{{n+i-1}} \right) \right|_{x}}{\omega_x}, \label{multidet} \end{equation}
for $f_1,\ldots,f_{n+i-1} \in \mathcal F^{(i-1)}$.
For convenience, define $\dot{\mathcal F}^{(i)}$ to be the vector space of functions on $\mathcal M^{(i)}$ which are of the form \eqref{multidet} only. One may also also regard $\mathcal F^{(i-1)}$ to be a subspace of $\mathcal F^{(i)}$ by composing with the projection $p_i$.

The manifolds $\mathcal M^{(N)}$ completely capture the analysis and geometry of the vector fields $Y_j^{(i)}$ and the function spaces $\mathcal F_N$ constructed in Lemma \ref{multisystem}. In a practical sense, this is because the problem of counting solutions can be lifted from $\mathcal M$ to $\mathcal M^{(N)}$. This idea is formalized by the following lemma.
\begin{lemma}
Suppose $\mathcal F_0$ consists of a finite-dimensional vector space of smooth functions on $\mathcal M$. Let $\mathcal F_1,\ldots,\mathcal F_N$ be the vector spaces of functions as constructed in the proof of Lemma \ref{multisystem}, i.e., $\mathcal F_i$ is the span of $\mathcal F_{i-1}$ and $Y_{j} \mathcal F_{i-1}$, $j=1,\ldots,n$, for vector fields $Y_j$ defined as in \eqref{vecdef} for some $f_{j_1},\ldots,f_{j_n} \in \mathcal F_{i-1}$. Then the number of nondegenerate solutions $x \in U$ of the system \label{count}
\begin{equation} f_1(x) = a_1, \ldots, f_n(x) = a_n, \label{systems0} \end{equation}
where $f_1,\ldots,f_n \in \mathcal F_N$, $a_1,\ldots,a_n \in \R$, for a given open set $U$ is equal to the number of nondegenerate solutions $p \in (\pi^{(N)})^{-1}(U)$ of a corresponding system
\begin{equation} F_1(p) = b_1,\ldots, F_{n+N}(p) = b_{n+N}, \label{bigsys} \end{equation}
where $F_1,\ldots,F_{n+N} \in \mathcal F^{(N)}$, $b_1,\ldots,b_{n+N} \in \R$.
\end{lemma}
Although the manifold $\mathcal M^{(N)}$ is somewhat more abstract than $\mathcal M$ itself, Lemma \ref{count} is a significant result for two reasons. The first is that it allows one to sidestep inherent difficulties of understanding the vector fields $Y_i$ when counting solutions. The second is that the functions in $\mathcal F^{(N)}$ are never more complex than derivatives of the functions in $\mathcal F$ and polynomials, as shown by the following proposition:
\begin{proposition}
Suppose that $\varphi$ is a diffeomorphism from some open set $U \subset \R^n$ onto some open subset of $\mathcal M$. For each $N$, there is a diffeomorphism $\varphi^{(N)}$ from $U \times \R_{\neq 0}^{N}$ onto $(\pi^{(N)})^{-1} (\varphi(U))$ such that for every  $F_1,\ldots,F_{n+N-1} \in \mathcal F^{(N-1)}$,
\begin{equation} 
\begin{split}
& \left. d^{n+N-1}(F_1,\ldots,F_{n+N-1}) \right|_{\varphi^{(N)}(x,t_1,\ldots,t_{N})} \\
 & \qquad \qquad \qquad =  t_1 \cdots t_{N} \det \frac{\partial (F_1,\ldots,F_{n+N-1})}{\partial (x,t_1,\ldots,t_{N-1})},
\end{split}
\label{coordform} 
\end{equation}
 where the determinant on the right-hand side is the usual Jacobian determinant in the coordinates $(x,t_1,\ldots,t_{N-1}) \in U \times \R^{N-1}_{\neq 0}$. \end{proposition}
\begin{proof}
By induction on $N$,  let $\varphi^{(N)}$ be given by
\[ \varphi^{(N)}(x,t_1,\ldots,t_N) := \left. \frac{d x_1 \wedge \cdots \wedge dx_n}{t_N} \wedge \frac{d t_1}{t_1} \wedge  \cdots \wedge \frac{d t_{N-1}}{t_{N-1}} \right|_{\varphi^{(N-1)}(x,t_1,\ldots,t_{N-1})}, \]
where $dx_1, \ldots, dx_n$ are differentials of the coordinate functions $x_1,\ldots,x_n$ on $\varphi^{-1}(U)$ induced by $\varphi$.
As can be seen from the formula, these coordinates have the property that the canonical projection from $\mathcal M^{(N)}$ to $\mathcal M^{(N-1)}$ corresponds to dropping the variable $t_{N}$.  It is easy to check in these coordinates that
\begin{align*}
d F_1 & \wedge \cdots \wedge d F_{n+N-1} \\ &
= \left[ \det \frac{\partial (F_1,\ldots,F_{n+N-1})}{\partial (x,t_1,\ldots,t_{N-1})} \right] dx_1 \wedge \cdots \wedge dx_n \wedge dt_1 \wedge \cdots \wedge dt_{N-1} \\
& = t_1 \cdots t_{N} \left[ \det \frac{\partial (F_1,\ldots,F_{n+N-1})}{\partial (x,t_1,\ldots,t_{N-1})} \right] \varphi^{(N)}(x,t_1,\ldots,t_{N})
 \end{align*}
 for any $F_1,\ldots,F_{n+N-1} \in \mathcal F^{(N-1)}$. Definition \eqref{multidet} immediately gives \eqref{coordform}.
\end{proof}
An important corollary is that when the functions $\mathcal F$ are polynomials of bounded degree in a suitable coordinate system (as will always be the case when applying the result to Lemma \ref{multisystem}), the functions $\mathcal F^{(N)}$ may also be regarded as polynomials of a suitably bounded degree in the appropriate coordinates as well. Thus the number of nondegenerate solutions to the system \eqref{bigsys} would immediately be bounded by B\'{e}zout's Theorem just as applied in the proof of Theorem \ref{radon}.

The proof of Lemma \ref{count} proceeds by showing that every function $f \in \mathcal F_N$ (the function space analogous to Lemma \ref{multisystem}) must agree with a function in $\mathcal F^{(N)}$ (the function space on $\mathcal M^{(N)}$) on a suitably-constructed $n$-dimensional submanifold of $\mathcal M^{(N)}$ which is defined implicitly via a system of equations in $\mathcal F^{(N)}$. This implies that the system of equations \eqref{systems0} involving the somewhat mysteriously-constructed functions $f_{j_1},\ldots,f_{j_n}$ can be naturally lifted to an system on $\mathcal M^{(N)}$ where the functions in the system belong to $\mathcal F^{(N)}$.
Because both $\mathcal F_N$ and $\mathcal F^{(N)}$ are vector spaces, the only part of this assertion which is somewhat cumbersome to prove is that ratios of wedge products {\it a la} \eqref{vecdef} appear as values of functions in $\mathcal F^{(N)}$ restricted to suitable submanifolds. This is accomplished by a trivial induction on $N$ combined with the following proposition, which shows how to identify quantities like \eqref{vecdef} via the identity \eqref{ratio} and also demonstrates in \eqref{snp1} how to inductively identify the $n$-dimensional submanifold of $\mathcal M^{(N)}$ on which the desired identities hold.
\begin{proposition}
Suppose $F_j \in \dot{\mathcal F}^{(j)}$ for each $j=1,\ldots,N$ and let \label{megacalc}
\[ {\mathcal M}^{(N)}_F :=  \set{ p \in \mathcal M^{(N)}}{ F_1(p) = \cdots = F_N(p) = 1}. \]
Then
\begin{enumerate}
\item The set ${\mathcal M}^{(N)}_F$ is a manifold and the projection $\pi^{(N)}$ is a diffeomorphism of any open subset of ${\mathcal M}^{(N)}_F$ and its image. 
\end{enumerate}
Next suppose that $h_1,\ldots,h_n$ and $g_1,\ldots,g_n$ are smooth functions on some open subset $O \subset \mathcal M$ for which there exist $H_1,\ldots,H_n, G_1,\ldots, G_n \in \mathcal F^{(N)}$ such that for each $j = 1,\ldots,n$, $H_j$ restricts to $h_j$ on ${\mathcal M}^{(N)}_F \cap (\pi^{(N)})^{-1} (O)$ and likewise for $G_j$ and $g_j$. In other words, $h_j \circ \pi^{(N)} = H_j$ on ${\mathcal M}^{(N)}_F \cap (\pi^{(N)})^{-1} (O)$ and $g_j \circ \pi^{(N)} = G_j$ on ${\mathcal M}^{(N)}_F \cap (\pi^{(N)})^{-1} (O)$ for each $j=1,\ldots,n$. If one defines
\begin{equation} F_{N+1} := d^{n+N} (G_1,\ldots,G_n,F_1,\ldots,F_N), \label{snp1} \end{equation}
the following must also be true:
\begin{enumerate}
\item[2.] The image $\pi^{(N+1)}( \mathcal M^{(N+1)}_F) \cap O \subset \mathcal M$ consists of exactly those points in $\pi^{(N)}(\mathcal M^{(N)}_F) \cap O$ at which $d g_1 \wedge \cdots \wedge d g_n \neq 0$.

\item[3.] There is a function in $\mathcal F^{(N+1)}$ which restricts to $$ \frac{d h_1 \wedge \cdots \wedge dh_n}{ d g_1 \wedge \cdots \wedge d g_n}$$ at every point of $O$ where the denominator is nonzero, namely
\begin{equation} \frac{ d h_1 \wedge \cdots \wedge dh_n}{d g_1 \wedge \cdots \wedge d g_n} \circ \pi^{(N+1)} = d^{n+N} ( H_1,\ldots,H_n, F_1,\ldots,F_{N}) \label{ratio} \end{equation}
on ${\mathcal M}^{(N+1)}_F \cap (\pi^{(N+1)})^{-1} (O)$.
\end{enumerate}
\end{proposition}
\begin{proof} 
 From the formula \eqref{coordform} in the coordinates $\varphi^{(N)}$ on $\mathcal M^{(N)} \cap (\pi^{(N)})^{-1}(U)$, it is clear that every $F_j \in \dot {\mathcal F}^{(j)}$ must equal $t_1 \cdots t_j$ times a polynomial in $(t_1,\ldots,t_{j-1})$ with coefficients that are smooth functions of $x$. There are several important consequences of this simple observation. The first is that $F_j$ is independent of $t_k$ when $k > j$. When $k = j$, it also follows that
 \begin{equation} \frac{\partial F_j}{\partial t_j} = \frac{1}{t_j} F_j. \label{tricalc} \end{equation}
 This means that the Jacobian matrix $\partial (F_1,\ldots,F_N) / \partial (t_1,\ldots,t_N)$ always has full rank at every point of $\mathcal M_F^{(N)}$ since the Jacobian matrix it is triangular and its diagonal entries are never zero (since $F_j = 1$ on $\mathcal M^{(N)}_F$ for each $j$ and by assumption $t_j \neq 0$ for each $j$ as well). By the Implicit Function Theorem, this guarantees that $\mathcal M^{(N)}_F$ is always a manifold regardless of the choice of the particular $F_j$'s. Moreover, because of this triangular structure and the linearity of $F_j$ as a function of $t_j$, it is easy to see that for a given $(x,t_1,\ldots,t_i) \in \mathcal M^{(i)}_F$, there is at most a unique value of $t_{i+1}$ such that $(x,t_1,\ldots,t_{i+1}) \in \mathcal M^{(i+1)}_F$, and such a solution exists if and only if $F_{i+1}(x,t_1,\ldots,t_i,t)$ is not an identically zero function of $t$. As already noted, if such a value of $t_{i+1}$ exists, it is necessarily true that the Jacobian determinant $\det \partial (F_1,\ldots,F_{i+1}) / \partial (t_1,\ldots,t_{i+1})$ must be nonvanishing at $(x,t_1,\ldots,t_{i+1})$. Therefore by the Implicit Function Theorem, the projection $\pi^{(N)}$ must be a diffeomorphism of any open subset of $\mathcal M^{(N)}_F$ and its image. This establishes the first conclusion of the proposition.
 
 Because $\pi^{(N)}$ is a diffeomorphism of any open subset of $\mathcal M^{(N)}_F$ and its image, one may define coordinates on $\mathcal M^{(N)}_F \cap (\pi^{(N)})^{-1}(U)$ using $\varphi$ by lifting the coordinate function $\varphi$ via $(\pi^{(N)})^{-1}$, i.e., by mapping $x \in U \cap \varphi^{-1} \pi^{(N)}(\mathcal M^{(N)}_F)$ to $(\pi^{(N)})^{-1} ( \varphi(x))$, where $U$ is any suitable open subset of $\mathcal M$ on which a coordinate system $\varphi$ is defined. Let $X_1,\ldots,X_n$ denote the associated coordinate vector fields. It follows that $d \pi^{(N)}(X_i) = \partial / \partial x_i$ for each $i=1,\ldots,n$. In the coordinates $\varphi^{(N)}$ on $\mathcal M^{(N)}$, this means that
 \[ X_i := \frac{\partial}{\partial x_i} + \sum_{j=1}^N c_{ij}(x,t) \frac{\partial}{\partial t_j} \]
 for each $i=1,\ldots,n$. Since each $F_j$ is constant on $\mathcal M^{(N)}_F$, it must be the case that $X_i F_j = 0$ on $\mathcal M^{(N)}_F$ for each pair of indices $i,j$. Therefore by applying the usual row operations to the Jacobian determinant \eqref{coordform} (assuming that distinct rows of the matrix correspond to partial derivatives with respect to distinct coordinate variables), it must be the case that
 \begin{align} d^{n+N}(G_1,\ldots,G_n,F_1,\ldots,F_N) & = t_1 \cdots t_{N+1} \left[ \det \frac{\partial G}{\partial X} \right] \left[ \det \frac{\partial F}{\partial (t_1,\ldots,t_N)} \right] \nonumber \\ & = t_{N+1} \left[ \det \frac{\partial G}{\partial X} \right]  \label{corecalc} \end{align}
 on $\mathcal M^{(N)}_F$ (using the triangular structure of $\partial F / \partial t$ and \eqref{tricalc}). If it is also known that $G_j$ restricts to $g_j$ on $\mathcal M^{(N)}_F \cap (\pi^{(N)})^{-1}(O)$, then $X_i G_j = X_i ( g_j \circ \pi^{(N)}) = (d \pi^{(N)}(X_i) g_j) \circ \pi^{(N)} = (\partial g_j / \partial x_i) \circ \pi^{(N)}$, so 
 \begin{equation} d^{n+N}(G_1,\ldots,G_n,F_1,\ldots,F_N) = t_{N+1}  \left[ \det \frac{\partial g}{\partial x} \right]  \label{coordform2} \end{equation}
 in the coordinates $(x,t_1,\ldots,t_{N+1})$ when $(x,t_1,\ldots,t_N) \in \mathcal M^{(N)}_F \cap (\pi^{(N)})^{-1}(U)$.
 
 Now assuming that $F_{N+1}$ is selected in such a way that \eqref{snp1} holds, it follows that for a given point $(x,t_1,\ldots,t_N) \in \mathcal M^{(N)}_F \cap (\pi^{(N)})^{-1}(U)$, the equation $F_{N+1}(x,t_1,\ldots,t_{N+1}) = 1$ will have a solution $t_{N+1}$ if and only if $\det (\partial g / \partial x) \neq 0$ at the point $x \in U$, which will occur exactly when $d g_1 \wedge \cdots \wedge d g_n \neq 0$. Because every point of $O$ is contained in an open set $U$ on which a coordinate system is defined, this forces the second conclusion of the proposition to be true, namely, that $\pi^{(N+1)}(\mathcal M_F^{(N+1)}) \cap O$ will be exactly the subset of $\pi^{(N)}(\mathcal M_F^{(N)}) \cap O$ at which $dg_1 \wedge \cdots \wedge dg_n \neq 0$.

 As for the third conclusion of the proposition, assuming that $x \in U$ is a point at which $d g_1 \wedge \cdots \wedge d g_n \neq 0$ and that $(x,t_1,\ldots,t_N) \in \mathcal M_F^{(N)}$, 
 \[ d^{n+N}( H_1,\ldots,H_n,F_1,\ldots,F_N) = t_{N+1} \left[ \det \frac{\partial h}{\partial x} \right] = \frac{\det \frac{\partial h}{\partial x}}{\det \frac{\partial g}{\partial x}} = \frac{d h_1 \wedge \cdots \wedge dh_n}{ d g_1 \wedge \cdots \wedge d g_n} \]
 assuming $1 = t_{N+1} \det ( \partial g / \partial x)$, which must be the case when $(x,t_1,\ldots,t_{N+1}) \in \mathcal M^{(N+1)}_F$. Because $U$ was arbitrary, the formula holds on all of $O$ as well.
 \end{proof}

The proof of Lemma \ref{count} follows quickly from Proposition \ref{megacalc}.
By induction on $N$, once it is known that there are suitable $F_i \in \dot{\mathcal F}^{(i)}$ for $i=1,\ldots,N$ such that every function $g \in \mathcal F_N$ of the form
$$Y^{(N)}_{j_N} \cdots Y^{(1)}_{j_1} f$$
for $f \in \mathcal F$ has a corresponding function $G$ in $\mathcal M^{(N)}$ which restricts to $g$  on $\mathcal M^{(N)}_F$, the third conclusion of the proposition establishes that the same property must hold at stage $N+1$ as well. This is because the functions $f_{j_1},\ldots,f_{j_n}$ in the denominator of \eqref{vecdef} defining the new vector fields $Y^{(N+1)}_i$ belong to the span of $\mathcal F_{N}$ and $Y_i^{(N)} \mathcal F_{N}$, which means by induction that each such function is the restriction to $\mathcal M^{(N)}_F$ of a function in $\mathcal F^{(N)}$. These extended functions define $F_{n+1}$ via \eqref{snp1}. The key point is that the vector fields $Y^{(N+1)}_1,\ldots,Y^{(N+1)}_n$ all have the same denominator, so the same choice of $F_{N+1}$ defining $\mathcal M^{(N+1)}_F$ works simultaneously for the application of any one of the vector fields $Y^{(N+1)}_i$ via the identity \eqref{ratio}.

A consequence of this observation is that when $H_1,\ldots,H_n \in \mathcal F^{(N)}$ restrict to $h_1,\ldots,h_n$ on some open subset of $\mathcal M_F^{(N)} \cap (\pi^{(N)})^{-1}(O)$, then every solution of the system of equations
\[ h_i(x) = a_i, \ \   i=1,\ldots,n, \]
for $x \in O$ will correspond to a solution of the augmented system
\[ H_i(x,t_1,\ldots,t_N) = a_i, \ \ i=1,\ldots,n, \mbox{ and } F_j(x,t_1,\ldots,t_N) = 1, \ \ j =1,\ldots,N, \]
in $\mathcal M^{(N)} \cap (\pi^{(N)})^{-1}(O)$ (in the sense that $(\pi^{(N)})^{-1}$ will map solutions in $O$ injectively to solutions in $\mathcal M^{(N)} \cap (\pi^{(N)})^{-1}(O)$ of the augmented system) and that the mapping preserves nondegeneracy in the sense that $\det (\partial h / \partial x) \neq 0$ for a solution point in $O$ if and only if $\det (\partial (H_1,\ldots,H_n,F_1,\ldots,F_N) / \partial (x,t_1,\ldots,t_N)) \neq 0$. This latter observation follows immediately from the equality of \eqref{coordform} (when fixing $(G_1,\ldots,G_{n+N}) := (H_1,\ldots,H_n,F_1,\ldots,F_N)$) and \eqref{coordform2}. Thus Lemma \ref{count} must be true. This completes the proof of Lemma \ref{count} and consequently the proofs of Lemma \ref{multisystem} and Theorems \ref{bestmeasthm} and \ref{betterthm} as well.

\section{Further applications to Radon-like operators}
\label{examples2}
To close, it is illuminating to return to the context of averaging operators \eqref{theop} of Theorem \ref{radon} and explicitly see how Theorem \ref{bestmeasthm} applies, as was abstractly indicated by Example 4 in Section \ref{examples1}. For convenience, it will be assumed that the map $\gamma(t,x)$ has the form
\[ \gamma(t,x) := (t , \gamma_0(t,x)) \]
where $\gamma_0 : \R^n \times \R^{N_2} \rightarrow \R^r$ for some integer $r$ (in which case $N_1 := n + r$) and $N_2 = rk$ for some integer $k \geq 2$. A short calculation gives that
\[ \omega(t,x) = (-1)^{nr} \sum_{1 \leq i_1 < \cdots < i_r \leq rk} \det  \left[ \! \begin{array}{ccc} \frac{\partial \gamma_0}{\partial x_{i_1}} (t,x) & \! \! \cdots & \! \! \frac{\partial \gamma_0}{\partial x_{i_r}} (t,x)  \end{array} \! \! \right] dx_{i_1} \wedge \cdots \wedge dx_{i_r} \]
because the determinants in the original definition \eqref{normalform} have block structure in the first $n$ rows and last $n$ columns. If the coordinates of $\gamma_0$ are labelled $(\gamma_0)_1,\ldots,(\gamma_0)_r$, then this formula for $\omega(t,x)$ agrees with the wedge product
\[ (-1)^{nr} d_x ( \gamma_0)_1 \wedge \cdots \wedge d_x (\gamma_0)_r, \]
where $d_x$ is the exterior derivative in the $x$ variables only. From this observation, it follows that $\Phi$ has the particularly simple form
\[ \Phi_x(t_1,\ldots,t_k) = (-1)^{nr} \det \left[ \! \!  \begin{array}{ccc} \left[ \frac{\partial \gamma_0}{\partial x} (t_1,x) \right]^T & \! \! \cdots &  \! \! \left[ \frac{\partial \gamma_0}{\partial x} (t_k,x) \right]^T \end{array} \! \! \right] \]
where $\partial \gamma_0 / \partial x$ is the $r \times rk$ Jacobian matrix of $\gamma_0$.

{\bf Example 1 (Hausdorff measure).}
Let $\mathcal C_\ell$ be the real associative algebra\footnote{The algebra $\mathcal C_\ell$ is an example of a Clifford algebra.} generated by elements $1,e_1,\ldots,e_\ell$ which are subject to the relations $1 e_j = e_j 1 = e_j$ for all $j$, $e_{i} e_j = - e_{j} e_i$ when $j \neq i$, and $e_i^2 = 1$. The dimension of the algebra as a vector space over the reals is $2^\ell$, and
\[ \left( \sum_{j=1}^\ell a_j e_j \right)^2 = \left( \sum_{j=1}^\ell a_j^2 \right) 1 \]
for any real numbers $a_1,\ldots,a_\ell$. Consequently if $M_1,\ldots,M_\ell$ are the $2^\ell \times 2^{\ell}$ matrices which express the action of left multiplication in $\mathcal C_\ell$ by $e_1,\ldots,e_\ell$, respectively, in the standard basis, then
\[ \det \left[ \sum_{j=1}^\ell a_j M_j \right]^2 = \left( \sum_{j=1}^\ell a_j^2 \right)^{2^\ell}. \]
If $n \leq \ell$ and one defines a mapping
\[ \Gamma(t) := \sum_{j=1}^\ell \Gamma_j(t) e_j  \]
for polynomial functions $\Gamma_1,\ldots,\Gamma_\ell$, then the Radon-like operator
\begin{equation} T f(y,x) := \int_{\R^n} f ( t, y + \Gamma(t) x) \chi_{\widetilde \Omega}(t,y,x) dt, \label{clifford} \end{equation}
where $x,y \in \mathcal C_\ell$, has the corresponding functional $\Phi$ 
\[ \Phi_{y,x} (t_1,t_2) = | \Gamma(t_2) - \Gamma(t_1)|^{2^\ell} \]
where $|\cdot|$ denotes the Euclidean distance of points in $\mathcal C$ when expressed in coordinates with respect to the standard basis.
This $\Phi$ vanishes to order $2^\ell$ on the diagonal, so when $\sigma = n 2^{-\ell}$ and $s = 2^\ell / n$ the optimal measure of Theorem \ref{bestmeasthm} is comparable to the $n$-dimensional Hausdorff measure on the image of $\Gamma$, assuming that $\Gamma(t)$ is locally injective.
If $\widetilde \Omega := \Omega \times \mathcal C_\ell \times \mathcal C_\ell$ for a set $\Omega$ on which $( \det (\partial \Gamma / \partial t)^T (\partial \Gamma / \partial t))^{1/2} \gtrsim \delta^{n/2^\ell}$, then \eqref{mainhyp} must apply and consequently
\begin{equation} ||T \chi_F ||_{L^{\frac{2^\ell + 2n}{n}}} \lesssim \delta^{-\frac{n}{2^\ell+2n}} |F|^{\frac{2n}{2^\ell + 2n}} \label{cliffordop} \end{equation}
for all Borel sets $F \subset \R^n \times \mathcal C_\ell$. In particular, note that the image of $\Gamma$ need not have any curvature whatsoever; in this case, the multiplicative structure of the Clifford algebra grants the operator \eqref{clifford} a sort of rotational curvature regardless of the higher-order geometric properties of $\Gamma$. If $\Gamma$ simply parametrizes a linear subspace, then \eqref{clifford} becomes a restricted $n$-plane transform; the estimate \eqref{cliffordop} can be taken to be global in $t$ and consequently scaling and Knapp examples give that the integrability exponents appearing in \eqref{cliffordop} are sharp.

{\bf Example 2 (Determinantal measure).} Generalizing the first example, suppose that $\Gamma : \R^n \rightarrow \R^{n' \times n'}$ is a polynomial map. The Radon-like operator
\begin{equation}
T f (y,x) := \int_{\R^n} f(t, y + \Gamma(t) x) \chi_{\widetilde \Omega}(t,y,x) dt \label{matrix} 
\end{equation}
where $y,x \in \R^{n'}$ and $\Gamma(t) x$ denotes matrix-vector multiplication, has functional
\[ \Phi_{y,x}(t_1,t_2) = \det (\Gamma(t_2) - \Gamma(t_1) ). \]
The order of vanishing $q$ of $\Phi$ on the diagonal must be at least $n'$. The associated measure $\mathcal H^{n/n'}_\Phi$ from Theorem \ref{bestmeasthm} is comparable to the $n/n'$-dimensional determinantal Hausdorff measure from Section \ref{detex} restricted to the image of $\Gamma$ (assuming, for example, that $\Gamma$ is locally injective). The measure must be absolutely continuous with respect to Lebesgue measure, so whenever it is nonzero, one can take $\widetilde \Omega := \Omega \times \R^{n'} \times \R^{n'}$ where $\Omega$ is any set on which the Radon-Nykodym derivative is at least comparable to $\delta^{n/n'}$. Then \eqref{mainhyp} will hold and the conclusion \eqref{mainconc} of Theorem \ref{radon} will hold with $k=2$ and $s = n'/n$. An extreme case occurs when $n = n'^2$ and $\Gamma$ is simply a linear isomorphism. Fixing $d T$ to Lebesgue measure on $\R^{n' \times n'}$, then the isodiametric determinantal inequality on $\R^{n' \times n'}$ proved in Proposition \ref{calcprop} implies the global, scaling-invariant inequality
\[ \left[ \int_{\R^{n'} \times \R^{n'}} \left|\int_{\R^{n' \times n'}} \chi_F(T, y + Tx) dT\right|^\frac{2n'+1}{n'} dx dy \right]^{\frac{n'}{2n'+1}} \lesssim |F|^{\frac{2n'}{2n'+1}} \]
for all Borel sets $F \subset \R^{n' \times n'} \times \R^{n'}$.

A modification of this example also applies to the case of convolution with measures on quadratic submanifolds of dimension $n$ in $\R^{2n}$. Specifically, fixing
\[ Q(a,b) := \left( \sum_{i,j=1}^n Q^1_{ij} a_i b_j, \ldots, \sum_{i,j = 1}^n Q^n_{ij} a_i b_j \right) \]
under the assumption that $Q^{\ell}_{ij} = Q^{\ell}_{ji}$ for each $i,j,\ell = 1,\ldots,n$, then
 the operator
\begin{equation} T f(y,x) = \int f(t, y - Q(x-t,x-t)) dt \label{average} \end{equation}
has a corresponding functional $\Phi$ given by
\[ \Phi_{y,x} (t_1,t_2) = \det ( Q(\cdot,t_2 - t_1)) \]
where $Q(\cdot,a)$ denotes the $n \times n$ matrix whose $(i,j)$-entry equals 
\[ \sum_{\ell=1}^n Q^i_{j \ell} a_\ell. \]
Since $\Phi$ is a polynomial of degree exactly $n$, the density \eqref{hausdens2} is a constant function. In the framework of geometric invariant theory, the infimum \eqref{hausdens2} is comparable to the infimum over the $\SL(n,\R)$-orbit of the polynomial $p(t) := \det Q(\cdot,t)$, where elements of $\SL(n,\R)$ act by linear coordinate changes (see, for example, the work of Richardson and Slodowy \cite{rs1990} extending the Kempf-Ness minimum vector construction to the context of real algebraic geometry). Thus the infimum is zero if and only if $p$ belongs to the nullcone of the representation. Because the nullcone is exactly the zero set of all $\SL(n,\R)$-invariant polynomials in the coefficients (which is a finitely generated algebra), this reduces the problem of applying Theorem \ref{radon} to \eqref{average} to a finite list of calculations once a set of generating $\SL(n,\R)$-invariant polynomials is known. This approach complements earlier work of the author \cite{gressman2015} which formulates a slightly weaker result in terms of the critical integrability exponent of the polynomial $\det Q(\cdot,t)$.

{\bf Example 3 (Affine measure).}
For the Radon-like operator
\begin{equation}
T f(x',x) := \int_{\R^n} f(t, x' + \Gamma(t) \cdot x) \chi_{\widetilde \Omega}(t,x',x) dt 
\label{radonaff} \end{equation}
where $x' \in \R$, $x \in \R^{k}$, and $\Gamma : \R^n \rightarrow \R^k$ is a polynomial map (and $\cdot$ is the dot product), the corresponding functional $\Phi$ equals
\[ \Phi_{x',x}(t_1,\ldots,t_k) = \det ( \Gamma(t_1) - \Gamma(t_{k+1}),\ldots, \Gamma(t_k) - \Gamma(t_{k+1})) \]
up to a factor of $\pm 1$. The order of vanishing $q$ must be at least $k$ but will generally be much larger. If $\sigma = n/q$ and $\Gamma$ is locally injective, then the sharp measure from Theorem \ref{bestmeasthm} is comparable to Oberlin's affine measure on the image of $\Gamma$; for general submanifolds, this measure will be comparable to affine submanifold measure as recently constructed by the author elsewhere \cite{gressman2017} (although the comparability may fail in special cases, e.g., when $\Gamma$ includes no mixed monomials). Unlike the Clifford algebra example, the nondegeneracy of affine submanifold measure on $\Gamma$ depends on higher-order geometry of $\Gamma$ and not just its first derivatives. Once again, because this measure is necessarily absolutely continuous with respect to Lebesgue measure, if the image of $\Gamma$ has nonzero affine Hausdorff measure, then a suitable $\widetilde \Omega$ can be defined to apply Theorem \ref{radon} to \eqref{radonaff}.

\bibliography{mybib}

\end{document}